\documentclass[12pt,english]{amsart}
\usepackage{amssymb,amsmath,amscd,graphicx,fontenc,amsthm,mathrsfs}
\usepackage[pdftex,bookmarks,colorlinks,breaklinks]{hyperref} 
\hypersetup{linkcolor=blue,citecolor=red}
\usepackage{multicol}
\usepackage[numbers]{natbib}
\usepackage{tikz}
\usetikzlibrary{matrix,cd}
\newtheorem{theorem}{Theorem}[section]
\newtheorem{lem}[theorem]{Lemma}
\newtheorem{proposition}[theorem]{Proposition}
\newtheorem{corollary}[theorem]{Corollary}

\theoremstyle{definition}

\newtheorem{remark}[theorem]{Remark}

\newtheorem{definition}[theorem]{Definition}

\numberwithin{equation}{section} 

\newcommand{\ZZ}{\mathbb{Z}}   
\newcommand{\CC}{\mathbb{C}}  

\newcommand{\U}{\mathrm{U}} 

\DeclareMathAlphabet{\mathbbold}{U}{bbold}{m}{n}
\def\bb1{\mathbbold{1}}
\def\bbz{\mathbb{Z}}

\def\bbr{\mathbb{R}}

\def\bbc{\mathbb{C}}

\def\bbn{\mathbb{N}}

\def\bbp{\mathbb{P}}

\def\ecal{\mathcal{E}}

\def\ucal{\mathcal{U}}

\def\ccal{\mathcal{C}}

\def\ncal{\mathcal{N}}

\def\cal{\mathcal{H}}
\def\lcal{\mathcal{L}}

\def\gfr{\mathfrak{g}}

\def\tfr{\mathfrak{t}}

\DeclareMathOperator\Aut{Aut}
\DeclareMathOperator\Inn{Inn}

\DeclareMathOperator\GL{GL}

\DeclareMathOperator\ad{ad}

\DeclareMathOperator\ind{ind}

\DeclareMathOperator\pr{pr}
\DeclareMathOperator\tr{Tr}

\DeclareMathOperator\supp{supp}

\def\h{\hspace{1mm}}

\def\vare{\varepsilon}
\def\be{\begin{equation}}
\def\ee{\end{equation}}

\newcommand{\wh}[1]{\widehat{#1}}

 \renewcommand{\d}[1]{\ \mathrm{d}#1}
\renewcommand{\H}{\mathscr{H}}
\newcommand{\op}{\mathrm{op}}
\newcommand{\cf}{\bb1}
\newcommand{\Dim}{\textrm{DC}}
\newcommand{\HS}{\textrm{HS}}

\usepackage[top=2cm,bottom=2cm,right=2cm,left=2cm]{geometry}
\title{Locally random groups}
\author[K. Mallahi-Karai]{Keivan Mallahi-Karai}
\address{ Jacobs University, Campus Ring I, 28759, Bremen, Germany}
\email{k.mallahikarai@jacobs-university.de}
\author[Amir Mohammadi]{Amir Mohammadi}
\address{Mathematics Department, University of California, San Diego, CA 92093-0112, USA }
\email{ammohammadi@ucsd.edu }
\author[Alireza Salehi Golsefidy]{Alireza Salehi Golsefidy}
\address{Mathematics Department, University of California, San Diego, CA 92093-0112, USA}
\email{golsefidy@ucsd.edu}
\begin{document}
\maketitle
\tableofcontents
\begin{abstract}
In this work, we will introduce and study the notion of {\it local randomness} for compact metric groups. We prove a mixing inequality as well as a product result for locally random groups under an additional {\it dimension condition} on the volume of small balls, and provide several examples of such groups. In particular, this leads to new examples of groups satisfying such a mixing inequality. In the same context, we will develop a Littlewood-Paley decomposition and explore its connection to the existence of spectral gap for random walks. Moreover, under the dimension condition alone, we will prove a multi-scale entropy gain result \`{a} la Bourgain-Gamburd and Tao.
\end{abstract}

\section{Introduction}\label{sec:intro}
The aim of this work is to introduce and study the notion of {\it{local randomness}} for the class of compact metric groups. As the name suggests, this notion aims at capturing a certain form of randomness exhibited by these groups. Before proceeding to the precise definition of this notion, let us make a few general remarks on the terminology and motivations behind the definition.

The notion of randomness is often understood as the lack of low-complexity structure. One approach towards defining randomness is {\it statistical randomness}. Roughly speaking, statistical randomness requires the putative random (sometimes called pseudo-random) object to pass certain randomness tests, which are passed by truly random objects.   Quasi-random graphs, introduced by Chung, Graham, and Wilson \cite{CGW89} are examples of this kind. For instance, in such a graph the number of edges connecting  subsets $A$, $B$ of vertices is close to $ \delta |A| \ |B|$, mimicking the typical behavior of Erd\"{o}s-R\'enyi random graphs with density $ \delta$.

An alternative approach towards defining randomness is based on the non-existence of low-complexity models. In taking up such an approach, one needs to clarify what a model means and how its complexity is measured. 
Quasi-random groups, as named by Gowers, provide examples for this approach. Recall that a finite group $G$ is said to be $K$-quasi-random when it admits no non-trivial unitary representations of degree less than $K$.
If one views a unitary representation of a finite group as a {\it model} and its degree as its complexity, then qausi-random groups are precisely groups without low-complexity models. 

One of the main results of Gowers's work, intertwining these two approaches, is that Cayley graphs of quasi-random groups with respect to {\it large} generating sets yield quasi-random graphs in the  sense of Chung, Graham, and Wilson. 
This is based on a mixing inequality established 
in~\cite{Gowers08}, and generalized in~\cite{BNP08}.
Let us remark that, prior to \cite{Gowers08}, the quasi-randomness had been implicitly exploited by 
Sarnak-Xue~\cite{Sarnak-Xue-91}  and Bourgain-Gamburd \cite{Bourgain-Gamburd-08}.

In the present work we will define the notion of local randomness for a compact group $G$ equipped with a compatible bi-invariant metric $d$ by means of an inequality of the form
\begin{equation}\label{eq:idea}
\|\pi(x)-\pi(y)\|_{\op}\le C_0 (\dim \pi)^L d(x,y) 
\end{equation}
where $C_0$ and $L$ are parameters and $\pi$ varies over unitary representation of $G$; see Definition \ref{def:local-random-with-metric} for the precise definition.
The relation between this inequality and the non-existence of low-complexity models for $G$ can be understood as follows. Consider an $\eta$-discretization of $G$, that is, a maximal set of points in $G$ that are pairwise $\eta$-apart. From \eqref{eq:idea} it follows that for a unitary representation $\pi$ of $G$ (a model) to map these points to matrices that are pairwise at distance $\eta^{1- \epsilon}$,  $\dim \pi$ needs to be polynomially large in $ \eta^{-1}$. Thus, a group satisfying \eqref{eq:idea} fails to have a low complexity discretized model. 


%
%

As the above definition indicates, local randomness of a compact group depends on the choice of a compatible metric. 
As we shall later see, when flexibility in the choice of metric is afforded, locally random groups can be characterized as those with finitely many non-equivalent irreducible representations of a given degree, see Theorem \ref{thm:finitely-many-representations}. 

Local randomness is much more fruitful when coupled with a {\em dimension condition}, see \eqref{eq:dimension-condition-intro}. 
In the presence of both properties, we will prove a local mixing inequality, Theorem \ref{thm:mixing-inequality}. This can be seen as an instance of statistical randomness and a multi-scale analogue of the mixing inequality alluded to above. This inequality enables us to prove a \emph{product result}, Theorem \ref{thm:product-large-subsets}, for subsets with large metric entropy, a result that can be best understood as a multi-scale version of Gowers's product theorem. 

In order to study the behavior of random walks on locally random groups, we adapt the Littlewood-Paley theory~\cite{Bourgain-Gamburd-2-08,BIG-17} to this context.
As an application, we will show that the study of spectral gap for random walks on $G$ can be reduced to that of {\it{functions living at small scale}}; see Theorem~\ref{thm:large-renyi-entropy} and Theorem~\ref{thm:functions-scale-spec-gap}.

Notable examples of groups to which our results apply include finite products of perfect real and 
$p$-adic analytic compact Lie groups. In the special case of profinite groups, local randomness is intimately connected to the notion of quasi-randomness introduced and studied in~\cite{Varju13}; see Proposition~\ref{prop:loc-rand-metric-quasi-rand} for precise statements. 
It is also worth mentioning that inequality \eqref{eq:idea} has been implicitly used in \cite{Saxce13} to establish the existence of a dimension gap for Borelean subgroups of compact Lie groups.

Our last theorem, Theorem \ref{thm:multi-scale-BG},  is an entropy gaining result in the spirit of a major ingredient of the Bourgain-Gamburd expansion machine. 
Roughly speaking,  this theorem asserts that when $X$ and $Y$ are independent $G$-valued random variables, the R\'{e}nyi entropy of $XY$ at scale $\eta$ is larger than the average of the R\'{e}nyi entropies of $X$ and $Y$ at scale $\eta$ by a definite amount, unless algebraic obstructions exist.
This can be viewed as a weighted version of Tao's result~\cite{Tao08} and a common extension of~\cite{Bourgain-Gamburd-2-08, Golsefidy-Varju-12,  Benoist-Saxce-16, BIG-17}. 

In a forthcoming work, we shall use Theorems \ref{thm:multi-scale-BG} and \ref{thm:functions-scale-spec-gap} in proving the \emph{spectral independence} of open compact subgroups of two non-locally isomorphic analytic simple Lie groups over local fields of characteristic zero.

This paper is structured as follows. In Section~\ref{sec:def}, we will review some basic definitions, set some notation and state the main results of the paper. In Section~\ref{sec:notation},  we gather a number of basic tools, ranging from abstract harmonic analysis to notions related to metric spaces. Sections~\ref{sec:thm1}
and \ref{sec:properties} feature prominent examples and fundamental properties of locally random groups.  In Section~\ref{s:mixing-inq}, we will prove a number of mixing properties for locally random groups, which will be employed in Section~\ref{sec:product} to show the product theorem. 
In Section~\ref{s:littlewoord}, we will discuss in detail a Littlewood-Paley decomposition of locally random groups. The connection to the spectral gap, stated in Theorem \ref{thm:large-renyi-entropy}, is established in Section~\ref{s:spec}. Finally, in Section~\ref{sec:BG}, we will prove Theorem \ref{thm:multi-scale-BG}. 


\section{Basic definitions and statement of results}\label{sec:def}
In this section, we will state the main results of the paper. Let us begin by defining the notion of local randomness.

\begin{definition}\label{def:local-random-with-metric}
	Suppose $G$ is a compact group and $d$ is a compatible bi-invariant metric on $G$. For parameters $C_0 \ge 1$ and $L \ge 1$ we say $(G,d)$ is $L$\emph{-locally random} with coefficient $C_0$ if for every irreducible unitary representation $\pi$ of $G$ and all $x,y\in G$ the following inequality holds:
	\be\label{eq:L-loc-rand-given-scale}
	\|\pi(x)-\pi(y)\|_{\op}\le C_0 (\dim \pi)^L d(x,y).
	\ee
	
We say a compact group $G$ is \emph{locally random} if $(G,d)$ is $L$-locally random with coefficient $C_0$ for some bi-invariant metric $d$ on $G$,  and some values of  $L$ and $C_0$.  
	\end{definition}
	
	

\begin{remark} 
\begin{enumerate}
\item It is a standard fact that every second countable compact group can be equipped with a compatible bi-invariant metric. 
\item One can easily check that \eqref{eq:L-loc-rand-given-scale} only depends on the unitary isomorphism class of $\pi$. 
\item In the rest of the paper, we will drop $d$ from the notation and use  the phrase {\em{$G$ is $L$-locally random with coefficient $C_0$}}. Often, the implicit metric $d$ is a natural metric on $G$. 
\end{enumerate}
\end{remark}

 Our first theorem gives a characterization of locally random groups in terms of their unitary dual.

\begin{theorem}[Characterization]\label{thm:finitely-many-representations}
Suppose $G$ is a compact second countable group. Then $G$ is locally random if and only if $G$ has only finitely many non-isomorphic irreducible representations of a given degree.
\end{theorem}

In \cite{BLM02} it is proved that a finitely generated profinite group has only finitely many irreducible representations of a given degree if and only if $G$ has the FAb property, that is, every open subgroup of $G$ has finite abelianization.


For  $\eta>0$ and $x\in G$,  denote the open ball of radius $\eta$ centered at $x$ by $x_\eta$. The $L^1$-normalized indicator function of the ball $1_{\eta}$ is denoted by 
$P_\eta:=\frac{\bb 1_{1_\eta}}{|1_\eta|}$, where  $| \cdot |$ denotes the Haar measure.  For  $f \in L^1(G)$ and a probability measure $\mu$ on $G$, we write $f_\eta=f*P_\eta$ and $\mu_\eta=\mu*P_\eta$, see \eqref{eq:def-conv-func} and \eqref{eq:def-conv-measure} for the definition of convolution.

\begin{definition}\label{def:dim-cond-intro}
Let $G$ be a compact group equipped with a compatible metric $d$. We say $(G,d)$ satisfies a dimension condition {$\Dim(C_1, d_0)$} if 
there exist $C_1 \ge 1$ and $d_0>0$ such that for all $\eta \in (0, 1 )$ the following bounds hold.
\be\label{eq:dimension-condition-intro}\tag{$\Dim$}
\frac{1}{C_1}\eta^{d_0} \le |1_{\eta}| \le C_1 \eta^{d_0}. 
\ee
\end{definition}

\begin{remark}
\begin{enumerate}
\item Measures satisfying this condition is also known as Ahlfors (or Ahlfors-David) $d_0$-regular measures. 
\item Whenever $d$ is clear from the context, we suppress $d$ from the notation and simply write that $G$ satisfies a dimension condition {$\Dim(C_1, d_0)$}. 
\end{enumerate}
\end{remark}
Our second theorem shows that local randomness is particularly effective in the presence of a dimension condition. 

\begin{theorem}[Scaled mixing inequality]\label{thm:mixing-inequality}
Suppose $G$ is an $L$-locally random group  with coefficient $C_0$. Suppose $G$ satisfies the dimension condition~\eqref{eq:dimension-condition-intro}.
Then for every $f,g\in L^2(G)$ we have 
\[
\|f\ast g\|_2^2\le 2 \|f_{\eta}\ast g_{\eta}\|_2^2+\eta^{1/(2L)} \|f\|_2^2\|g\|_2^2
\]
so long as $C_0\sqrt\eta\leq 0.1$.
\end{theorem}

Similar statements for finite groups, simple Lie groups and perfect Lie groups have been established thanks to work of many authors, see e.g.~\cite{Gowers08, BNP08, Bourgain-Gamburd-12, BIG-17,Saxce-Benoist-16}. 

\begin{definition}\label{def:metric-entr}
Suppose $X$ is a metric space and $A \subseteq X$. For $\eta\in (0,1)$, $\ncal_{\eta}(A)$ denotes the least number of open balls of radius $\eta$ with centers in $A$ required to cover $A$. The metric entropy of $A$ at scale $\eta$ is defined by
\[
h(A;\eta):=\log \ncal_{\eta}(A).
\] 
\end{definition}

\begin{theorem}[Product theorem for locally random groups]\label{thm:product-large-subsets}
Suppose $G$ is an $L$-locally random group with coefficient $C_0$. Suppose $G$ satisfies the dimension condition~{\emph{\ref{eq:dimension-condition-intro}}}$(C_1, d_0)$.
Then for every $\vare>0$, there exists $\delta>0$ such that for all $\eta>0$ and $A,B\subseteq G$ satisfying 
	\[
	 \frac{h(A;\eta)+h(B;\eta)}{2}\ge (1-\delta)h(G;\eta)
	\]
	and $\eta^{\vare}\ll_{L,C_0,C_1,d_0} 1$, 
	we have 
	\[
	A_{\eta}B_{\eta}B_{\eta}^{-1}A_{\eta}^{-1}\supseteq 1_{\eta^{\vare}}.
	\]
\end{theorem}

\begin{definition}
Suppose $G$ is a compact group and $\mu$ is a symmetric Borel probability measure. Denote by $T_\mu$ the convolution operator on $L^2(G)$ mapping $f$ to $ \mu \ast f$. 
For a subrepresentation $(\pi, \cal_\pi)$ of 
$L^2_0(G)$, we let $$\lambda(\mu;\cal_\pi):= \|T_{\mu}|_{\cal_{\pi}}\|_{\op} \quad \text{and} \quad \lcal(\mu;\cal_\pi):=-\log \lambda(\mu;\cal_\pi).$$
\end{definition}

Given a $G$-valued random variable $X$, we define the \emph{R\'{e}nyi entropy of $X$ at scale} $\eta$ by
\[
H_2(X;\eta):=\log(1/|1_{\eta}|)-\log \|\mu_{\eta}\|_2^2,
\]
where $\mu$ is the distribution (or the law) of $X$. 

\begin{theorem}\label{thm:large-renyi-entropy}
Suppose $G$ is an $L$-locally random group with coefficient $C_0$. Suppose $G$ satisfies~{\emph{\ref{eq:dimension-condition-intro}}}$(C_1, d_0)$.
Then there exist $ \eta_0>0$ small enough depending on the parameters and a subrepresentation $\cal_0$ {\emph{(exceptional subspace)}} of $L^2(G)$ such that the following statements hold.
\begin{enumerate}
\item {\emph{(dimension bound)}} $\dim \cal_0\le 2C_0 \eta_0^{-d_0}$.
\item {\emph{(spectral gap)}} Let $\mu$ be a symmetric Borel probability measure whose support generates a dense subgroup of $G$. Let  $a>\max(4Ld_0,4L+2)$, and for $i \ge 1 $ set $\eta_i:=\eta_0^{a^i}$. 
If for constant $C_2>0$ and for every positive integer $i$ there exists an integer $l_i\le C_2 h(G;\eta_i)$ such that
	\[
	(\text{Large entropy at scale } \eta_i) \hspace{1cm} H_2(\mu^{(l_i)};\eta_i) \ge  \Bigl( 1-\frac{1}{20Ld_0a^3} \Bigr)  h(G;\eta_i),
	\]
	then  
	\[
	\lcal(\mu; L^2(G)\ominus \cal_0)\ge \frac{1}{40C_2 Ld_0a^3}.
	\]
	In particular, $\lcal(\mu;L^2_0(G))>0$.  
\end{enumerate}
\end{theorem}

Finally, we prove a multi-scale entropy gain result which is in the spirit of \cite[Lemma 2.1]{BG-JEMS2} by Bourgain and Gamburd,  and is a weighted version 
of \cite[Theorem 6.10]{Tao08} by Tao. 
More details on the background of this result will be mentioned in Section~\ref{sec:BG}. Before we state this result, we recall the definition of an approximate subgroup.   
\begin{definition}
A subset $X$ of a group $G$ is called a $K$-\emph{approximate subgroup} if $X$ is a symmetric subset, that is, $X=X^{-1}$, and there exists subset $T \subseteq X\cdot X$ such that $\# T \le K$ and $X\cdot X\subseteq T \cdot X$.
\end{definition}
 
\begin{theorem}\label{thm:multi-scale-BG}
Suppose $G$ is a compact group which satisfies the dimension condition at scale $\eta$, that is, 
\[
C^{-1} \eta^{d_0} \le |1_{a\eta}|\le C \eta^{d_0}
\]	
holds for all $a\in [C'^{-1},C']$, where $C>1, C' \gg 1, d_0 >0$ are fixed numbers. Suppose $X$ and $Y$ are independent Borel $G$-valued random variables. If
\[
H_2(XY;\eta)\le \log K+\frac{H_2(X;\eta)+H_2(Y;\eta)}{2}
\]
for some positive number $K\ge ( C 2^{d_0})^{O(1)} $, then there are $H\subseteq G$ and $x,y\in G$ satisfying the following properties:
\begin{enumerate}
\item {\emph{(Approximate structure)}} $H$ is an $O(K^{O(1)})$-approximate subgroup.
\item {\emph{(Metric entropy)}} $|h(H;\eta)-H_2(X;\eta)|\ll \log K$.
\item {\emph{(Almost equidistribution)}} Let $Z$ be a random variable with the uniform distribution over $1_{3\eta}$ independent of $X$ and $Y$. Then 
\[\bbp(XZ\in (xH)_{\eta}) \ge K^{-O(1)} \text{ and } \bbp(YZ\in (Hy)_{\eta}) \ge  K^{-O(1)}.\]
Moreover, 
\[| \{h\in H_{\eta} |\h \bbp(X \in (xh)_{3\eta}) \ge \widehat{C}  K^{-10} 2^{-H_2(X;\eta)}\}  | \ge  K^{-O(1)} | H_{\eta} |, 
\]
where $\widehat{C}$ is a constant of the form $(C2^{d_0})^{O(1)}$. 
\end{enumerate}
\end{theorem}


\section{Preliminaries and notation}\label{sec:notation}
The purpose of this section is to provide the necessary definitions and fix the notation for the rest of the paper. For reader's convenience, these have been 
organized in two subsections. 

Let $G$ be a compact Hausdorff second countable topological group. It is well known that $G$ can be equipped with a bi-invariant metric that induces the topology of $G$. Moreover, there exists a unique bi-invariant  probability measure defined on the Borel $\sigma$-algebra of $G$, called the Haar measure.
For a Borel measurable subset $A$ of $G$, the Haar measure of $A$ is denoted by $m_G(A)$ or $|A|$.  For a Borel measurable function $f: G \to \CC$, the integral of $f$ with respect to  the Haar measure is denoted, interchangeably, by  $\int_G f$ or $ \int_G f(y) \d{y}$. 
We denote by $L^p(G)$ the space (of equivalence classes) of complex-valued functions $f$ on $G$ satisfying $\int_G |f(x)|^p \d{x} < \infty$. For $ f\in L^p(G)$, 
we write $$\| f \|_p = \Big(  \int_G |f(x)|^p \d{x} \Big)^{1/p}.$$  We will also denote by $C(G)$ the Banach space of complex-valued continuous functions $f: G \to \CC$, equipped with the supremum norm. 
For $f, g \in L^1(G)$ the convolution $f \ast g$ is defined by 
\begin{equation}\label{eq:def-conv-func}
  ( f \ast g) (x)= \int_G f(y) g( y^{-1} x) \d{y}. 
\end{equation}
It is a fact that $(L^1(G), +, \ast)$ is a unital Banach algebra and if $f \in L^1(G)$ is a class function, then $f$ is in the center of this Banach algebra.  
Note also that $L^2(G)$ is naturally equipped with the inner product defined by $ \langle f,  g \rangle = \int_G f \overline{g}$ is a Hilbert space. 

When $\H$ is a Hilbert space and $T: \H \to \H$ is a bounded linear operator, we will define the operator norm of $T$ by 
\[ \| T \|_{\op}= \sup_{ v \in \H \setminus \{ 0 \}  } \frac{ \| Tv \|}{ \| v \|}. \]
When $\H$ is finite-dimensional, the Hilbert-Schmidt norm of $T$ is defined by 
\[ \| T \|_{\HS} = ( \tr (  TT^{\ast}) ) ^{1/2}, \]
where $T^{\ast}$ denotes the conjugate transpose of the operator $T$. Note that when $S$ and $T$ are linear operators on a finite-dimensional Hilbert space $\H$, the following 
inequality holds 
\[ \| TS \|_{\HS} \le \| T \|_{\op}  \|  S \|_{\HS}. \]

For a Hilbert space $\H$, we write $\U(\H)$ for the group of unitary operators of $\H$.  
A homomorphism $ \pi: G \to \U(\H)$ is continuous if the map $$G \times \H \to \H, \qquad (g, v) \mapsto g \cdot v$$ is continuous.
A unitary representation of $G$ (or sometimes called a $G$-representation) is a pair $(\H, \pi)$ consisting of a Hilbert space $\H$
and a continuous homomorphism $\pi: G \to  \U(\H)$. 
A closed subspace $\H' \subseteq \H$ is called $G$-invariant (or simply invariant when $G$ is clear from the context) if for every $ g \in G$ and every $ v \in \H'$, one has $ g \cdot v \in \H'$. A representation $(\H, \pi)$ is called irreducible when $\dim \H \ge 1$ and the only invariant subspaces are $\{ 0 \}$ and $\H$ itself. The set of equivalence classes
of irreducible unitary representations of $G$ is called the unitary dual of $G$ and is denoted by $ \widehat{G}$. If $ \H'$ is an invariant subspace of $\H$, 
we sometimes denote by $ \H \ominus \H'$ the orthogonal complement of $\H'$ in $H$, which is itself an invariant subspace of $ \H$.  The 
set of vectors $v \in \H$ satisfying $\pi(g)v =v$ for all $ g \in G$ is clearly a closed invariant subspace of $\H$ and is denoted by $\H^G$. 

The group $G$ acts on $L^2(G)$  via $ (g \cdot f)( x) = f( g^{-1}x)$, preserving the $L^2$-norm. Hence, it defines a unitary representation of $G$ on $L^2(G)$, 
which is called the regular representation of $G$. 

Suppose $\mu$ and $\nu$ are Borel measures on $G$ and $ f \in L^1(G)$. The convolution $ \mu \ast f$ is defined by   
\begin{equation}\label{eq:def-conv-measure}
(\mu \ast f)(x) = \int_G f(y^{-1} x) \d{\mu}(y). 
\end{equation}
Similarly, the convolution $ \mu \ast \nu$ is the probability measure on $G$ is defined through its action on continuous functions via
\[ \int_G f \d{(\mu \ast \nu)} = \int_G \int_G f(xy) \d\mu(x) \d{\nu}(y),  \]
where $f \in C(G)$.  The following special cases of Young's inequality for $f, g \in L^2(G)$ and probability measure $\mu$  will be freely used in this paper:
\be\label{eq:Young-ineq}
\|  f \ast g   \|_2 \le \| f \|_1 \ \| g \|_2, \quad    \| f \ast g  \|_{\infty}   \le \| f \|_2 \ \| g \|_2, \quad \| \mu \ast f \|_2 \le \| f \|_2.
\ee

Let us enumerate a number of well-known facts about unitary representations of $G$. First, it is known 
that every $ \pi \in \widehat{G}$ is of finite dimension, and that every unitary representation of $G$ can be decomposed
as an orthogonal direct sum of $ \pi \in \widehat{G}$. A function $f \in  L^2(G)$ is called $G$-finite if there exists a finite-dimensional $G$-invariant subspace of $L^2(G)$ containing $f$. It is clear that 
$G$-finite functions form a  subspace of $L^2(G)$. We will denote this subspace by $\mathcal{E}(G)$. It follows from the classical theorem of Peter-Weyl that $\mathcal{E}(G) \subseteq C(G)$
and that $\mathcal{E}(G)$ is dense in $L^2(G)$.

For  $\pi \in \widehat{G}$ and $f \in L^1(G)$, the Fourier coefficient $ \widehat{f}(\pi)$ is defined by 
\[ \widehat{f}(\pi)= \int_G f(g) \pi(g)^{\ast} \d\mu(g). \]
One can show that for $f, g \in L^1(G)$ and $\pi \in  \widehat{G}$, we have
\[ \widehat{f \ast g} (  \pi)=  \widehat{g}( \pi) \widehat{f}(\pi). \]
Parseval's theorem states that for all $f \in L^2(G)$ the following identity holds:
\[ \| f \|_2^2= \sum_{\pi\in \widehat{G} } \dim \pi \h \|\widehat{f}(\pi)\|_{\HS}^2. \]

Finally, we will remark that $G$ is abelian if and only if every $\pi \in  \widehat{G}$ is one-dimensional. 
In this case, the above discussion reduces to the classical Fourier analysis on abelian groups. 

In this subsection, we will collect a number of definitions from additive combinatorics that will be needed later. Let $G$ be as above, and recall that $d$ denotes a bi-invariant metric on $G$.  The ball of radius $ \eta>0$ centered as $x \in G$ is denoted by $x_{ \eta}$. The $ \eta$-neighborhood of a set $A$, denoted by $A_\eta$, is the union of all $x_{ \eta}$ with $x \in A$.    

A subset $A \subseteq G$ is said to be  $\eta$-separated if the distance between every two points in $A$ is at least $ \eta$. An $\eta$-cover for $A$ is a collection of balls of radius 
$\eta$ with centers in $A$ whose union covers $A$.  Recall that the minimum size of an $\eta$-cover of $A$ (which is finite by compactness of $G$) is denoted by $\ncal_{\eta}(A)$.
The value $$ h(A;\eta):=\log \ncal_{\eta}(A)$$ is called the metric entropy of $A$ at scale $\eta$. 

The characteristic function of a set $A$ is denoted by $\cf_A$. For $ \eta >0$, we write $P_{ \eta}= \frac{\cf_{1_{\eta} }}{|1_{\eta}|}$. Note that $P_\eta$ belongs to the center
of the Banach algebra $L^1(G)$.  For $ f\in L^1(G)$ ($\mu$ probability measure on $G$, respectively) we write $f_{\eta}$ ($\mu_\eta$, respectively) instead
of $f \ast P_{\eta}$ ($\mu \ast P_{\eta}$, respectively).  The cardinality of a finite set $A$ is denoted by $\# A$. 
The R\'{e}nyi entropy of a $G$-valued Borel random variable $X$ at scale $\eta>0$ is defined by 
\[
H_2(X;\eta):=\log(1/|1_{\eta}|)-\log \|\mu_{\eta}\|_2^2,
\]
where $\mu$ is the distribution measure of $X$. As $H_2(X;\eta)$ depends only on the distribution measure $\mu$ of $X$, we will sometimes write $H_2(\mu;\eta)$ instead of
$H_2(X;\eta)$. 

We will use Vinogradov's notation $A \ll_{c_1, c_2} B$ to denote that $A \le CB$, where $C=C(c_1, c_2)$ is a positive function of $c_1, c_2$. We write  
$A \ll B$ to denote that $A \le CB$, for some absolute constant $C>0$. We similarly define $\gg_{c_1, c_2}$ and $\gg$ for the reverse relations.

 
 \section{Local randomness and representations with bounded dimension}\label{sec:thm1}
The main goal of this section is to prove Theorem~\ref{thm:finitely-many-representations}. Along the way some basic properties of locally random groups will also be proved. 

Suppose $f:\bbz^+\rightarrow \bbr^+$ is a strictly increasing function, and define 
\be\label{eq:metric}
d_{G,f}(x,y):=\sup_{\pi\in \wh{G}} \frac{\|\pi(x)-\pi(y)\|_{\op}}{f(\dim \pi)}.
\ee
Note that $\frac{\|\pi(x)-\pi(y)\|_{\op}}{f(\dim \pi)}$ depends only on the (unitary) isomorphism class of $\pi$. In the sequel 
we often assume $\pi: G\to U(n)$ for some $n\in\bbn$.
Moreover, we remark that if $\pi$ is a finite dimensional unitary representation of $G$ with the orthogonal decomposition $\pi=\oplus_{i \in I} \pi_i$ into 
irreducible representations, then
\be\label{eq:dGf-compl-red}
\frac{\|\pi(x)-\pi(y)\|_{\op}}{f(\dim \pi)} \le \max_{i \in I} \frac{\|\pi_i(x)-\pi_i(y)\|_{\op}}{f(\dim \pi_i)}  \le  d_{G, f} (x,y).  
\ee

\begin{lem}\label{lem:getting-metric}
Suppose $G$ is a compact group and  $f:\bbz^+\rightarrow \bbr^+$ is a strictly increasing function. 
Let $d_{G,f}$ be defined as in \eqref{eq:metric}; then $d_{G,f}$ is a well-defined bounded, bi-invariant metric on $G$.	
\end{lem}
\begin{proof}
	Since $\|\pi(x)\|_{\op}=1$ for all $\pi\in\wh{G}$ and all $x\in G$,  
	we have that $d_{G,L}(x,y)\le 2/f(1)$ for any $x,y\in G$ ---we also used the fact that $f$ is increasing.
As $\pi(z)$ is a unitary matrix for any $z\in G$,
$$\|\pi(x)-\pi(y)\|_{\op}=\|\pi(zx)-\pi(zy)\|_{\op}=\|\pi(xz)-\pi(yz)\|_{\op}.$$
This implies $d_{G,f}$ is bi-invariant. 
Clearly $d_{G,f}$ satisfies the triangle inequality. By the Peter-Weyl theorem, if $x\neq y$, then there is $\pi\in \wh{G}$ such that $\pi(x)\neq \pi(y)$. Hence, if $x\neq y$, then $d_{G,f}(x,y)\neq 0$, from which the claim follows.
\end{proof}
Next we want to explore the conditions under which the metric $d_{G,f}$ gives us the same topology as the original topology of $G$. 
Indeed it suffices to study neighborhoods of the identity. 

\begin{lem}\label{lem:towards-def}
In the above setting, $d_{G,f}$ induces the original topology of $G$ if and only if 
\[\lim_{x\rightarrow 1} d_{G,f}(x,1)=0.\]
\end{lem}

\begin{proof}
In order to distinguish the two topologies on $G$, we let $G_f$ denote the topological space whose point set is $G$ and whose topology is generated by the metric $d_{G,f}$. 

If $G$ and $G_f$ coincide, then $\lim_{x\rightarrow 1}d_{G,f}(x,1)=0$.

Conversely, let $I_G:G\rightarrow G_f$ be the identity map. 
Since $d_{G,f}$ is bi-invariant, $\lim_{x\rightarrow 1} d_{G,f}(x,1)=0$ implies $\lim_{x\rightarrow y} d_{G,f}(x,y)=0$ for all $y\in G$. Hence $I_G$ is continuous. Since $G$ is compact and $I_G$ is a continuous bijection, it is a homeomorphism; this finishes the argument.
\end{proof}

The following is a generalization of Theorem~\ref{thm:finitely-many-representations}. 

\begin{theorem}\label{thm:general-finite-number-irreducible-rep}
Suppose $G$ is a compact group. The following statements are equivalent.
\begin{enumerate}
\item For any strictly increasing function $f:\bbz^+\rightarrow \bbr^+$, the metric $d_{G,f}$ induces the original topology of $G$.
\item For some strictly increasing function $f:\bbz^+\rightarrow \bbr^+$, the metric $d_{G,f}$ induces the original topology of $G$.
\item For any positive integer $n$, $\{\pi\in \wh{G}|\h \dim \pi\le n\}$ is finite. 	
\end{enumerate}	
\end{theorem}

We start by proving that the second condition implies the \emph{FAb condition}.
\begin{definition}
A compact group $G$ has the \emph{FAb property} if $H^{\rm ab}:=H/\overline{[H,H]}$ is finite for any open subgroup $H$ of $G$. 	
\end{definition}

\begin{lem}\label{lem:TheFAb}
	Suppose $G$ is a compact group and for some strictly increasing function $f:\bbz^+\rightarrow\bbr^+$, the metric 
	$d_{G,f}$ induces the original topology of $G$. Then $G$ has the FAb property.
\end{lem}
\begin{proof}
We first show that $N/\overline{[N,N]}$ is finite for any {\em normal} open subgroup $N$ of $G$. 

 Let $\xi\in\wh{N}$, $\dim\xi=1$, and $n\in \bbz^+$, then $x\mapsto \xi(x)^n$ also defines a character of $N$; we denote this charatcter by $\xi_n\in \wh{N}$.	

Let $\pi_n:=\ind^G_N(\xi_n)$. We can identify the space $\cal_{\pi_n}$ of the representation $\pi_n$ with 
\[
\bbc[G]\otimes_{\bbc[N]} \cal_{\xi_n}=\bigoplus_{j=1}^m \bbc (g_j\otimes 1)
\]
where $\bbc[G]$ and $\bbc[N]$ are the corresponding group rings and 
$\{g_j\}_{j=1}^m$ is a set of coset representatives of $N$ with $g_1=1$; note that the inner product 
is induced from $\langle g_i\otimes 1,g_j\otimes 1\rangle =\delta_{ij}$, and for any $y\in N$ we have 
\[
\pi_n(y)(g_j\otimes 1)=g_j\otimes \xi_n(g_j^{-1}yg_j)=\xi(g_j^{-1}yg_j)^n (g_j\otimes 1).
\]
Therefore, $\|\pi_n(y)-I\|_{\op}=\max_j|\xi(g_j^{-1}yg_j)^n-1|\ge |\xi(y)^n-1|$. 
By Lemma~\ref{lem:towards-def}, we get that for any $\vare>0$ there exists $\eta>0$ with the following property: 
for any $\xi\in \wh{N}$ with dimension 1 and any $n\in\bbz^+$ we have
\be\label{eq:degree-1}
|\xi(y)^n-1|\le f([G:N]) \vare\quad\quad\text{for all $y\in N\cap 1_{\eta}$}
\ee
see~\eqref{eq:dGf-compl-red}.

Note that if $\zeta\in\mathbb S^1\setminus\{1\}$ is a norm 1 complex number that is not 1, then there is a positive integer $n$ such that $|\zeta^n-1|\ge \sqrt{3}$. Hence, \eqref{eq:degree-1} implies that if $\vare<\sqrt{3}/f([G:N])$, then $\xi(x)=1$ for $x\in N\cap 1_{\eta}$. Therefore, there is $\eta>0$ such that $1_{\eta}\subseteq \ker \xi$ for all $\xi\in \wh{N}$ that has dimension 1; thus, 
$\overline{[N,N]}= \bigcap_{\xi\in \wh{N}, \dim \xi=1} \ker \xi$ is an open subgroup of $G$. In particular, $N/\overline{[N,N]}$ is finite.

Suppose now that $H$ is an arbitrary open subgroup of $G$; 
then $G$ acts on the finite set $G/H$ by the left multiplication. The kernel $N$ of this action is an open normal subgroup of $G$. 
Since $\overline{[N,N]}\subseteq \overline{[H,H]}$ and $\overline{[N,N]}$ is an open subgroup, the claim follows.
\end{proof}

\begin{lem}\label{lem:SubgroupGrowth}
	Suppose $G$ is a compact group and $d_{G,f}$ induces the original topology of $G$ for some strictly increasing function $f:\bbz^+\rightarrow\bbr^+$. Then $G$ has only finitely many open subgroups of index at most $n$ for any positive integer $n$.
\end{lem}
\begin{proof}
Suppose to the contrary that $G$ has infinitely many open subgroups $\{H_i\}_{i=1}^{\infty}$ of index at most $n$. Let $\pi_i$ be the representation of $G$ on $L^2(G/H_i)$. Since $\pi_i(x)$ is a permutation for any $x\in G$, we have $\|\pi_i(x)-I\|_{\op}\ge 1$ for $x\not\in \ker \pi_i$. And so for any $x\not\in N:=\bigcap_{i=1}^{\infty} \ker \pi_i$ we have $d_{G,f}(x,1)\ge 1/f(n)$. Therefore by Lemma~\ref{lem:towards-def}, we have that $N$ is an open subgroup of $G$. 

Since $\ker \pi_i\subseteq H_i$, the sequence $\{H_i/N\}_{i=1}^{\infty}$ consists of distinct subgroups of a finite group $G/N$, which is a contradiction. 
\end{proof}

Next we prove a lemma on compact Lie groups which is essentially due to Platonov~\cite{Platonov66}, and its proof has some similarities with the proof of the Schur-Zassenhaus theorem in finite group theory. We will provide a proof for the convenience 
 of reader. 

\begin{lem}\label{lem:Platonov}
Let $G$ be a closed subgroup of the unitary group $\U_n(\CC)$. Then there exists a finite subgroup $F \le G$ such that $G= FG^{\circ}$. 
\end{lem}

\begin{proof}
Let $T$ be a maximal torus in $G^{\circ}$. For every $g \in G$, the torus $T^g := gTg^{-1}$ is also a maximal torus in $G^{\circ}$, and hence, by conjugacy of maximal tori, there exists $g_0 \in G^{ \circ}$ such that $T^{g}= T^{ g_0}$, or, equivalently, $ g_0^{-1}g \in N_G(T)$. This establishes that $G= N_G(T) G^{\circ}$. In view of the fact that $\Aut(T) \simeq \GL_d(\ZZ)$ is discrete, we have $N_G(T)^{ \circ} \subseteq C_G(T)$. Moreover, as $T$ is a maximal torus, we have $C_G(T)^{\circ} = T$, and hence $N_G(T)^{ \circ} = T$. From the compactness
of $N_G(T)$, we obtain that $[N_G(T): N_G(T)^{ \circ}]< \infty$, and hence we have the following exact sequence:
\[ 1 \to T \to N_G(T) \to \overline{F} \to 1, \]
where $ \overline{F} $ is a finite group. Since $T$ is abelian, the conjugation action of $N_G(T)$ on $T$ induces an action of $ \overline{F} $ on $T$. For any $f\in  \overline{F} $ and $t\in T$, we denote the action of $f$ on $t$ by $t^f$. That means if $s:  \overline{F}  \to N_G(T)$ is a section for the projection map from $N_G(T)$ to $ \overline{F} $, then $t^f=s(f)ts(f)^{-1}$ for any $f\in  \overline{F} $ and $t\in T$. For the section $s$, let
$c(f_1, f_2)= s( f_1 f_2)^{-1} s(f_1) s(f_2)$ for $ f_1, f_2 \in  \overline{F} $. Note that $c(  \overline{F}  \times  \overline{F} ) \subseteq T$, and we have 
\[   s( f_1f_2) c(f_1,f_2) = s(f_1) s(f_2). \]
From here one can verify the following $2$-cocycle relation:
\begin{equation}\label{cocycle}
  c(f_1, f_2f_3) c(f_2, f_3)= c(f_1f_2, f_3)^{ f_3^{-1}} c(f_1, f_2).
\end{equation}
Let $  \alpha:  \overline{F}  \to T$ be defined in such a way that $ \alpha(f)^{\#  \overline{F} } = \prod_{f' \in  \overline{F}  }  c(f', f)$. From \eqref{cocycle} and the definition of $ \alpha$ we have 
\begin{equation}\label{eq:cocycle}
\begin{split}
 \alpha( f_1, f_2)^{\#  \overline{F} } & = \prod_{f' \in  \overline{F}  }  c( f', f_1 f_2) = \prod_{f' \in  \overline{F}  }   \left( c( f'f_1, f_2)^{ f_2^{-1} } c( f', f_1) c( f_1, f_2)^{-1} \right) \\
 &=  \left(  \alpha( f_2) ^{ f_2^{-1}} \alpha( f_1) c( f_1, f_2) \right)^{\#  \overline{F} }. 
\end{split}
\end{equation}
Let $T_{\overline{F}}$ be the subgroup of $T$ consisting of all elements of order dividing $\#  \overline{F} $;  then from \eqref{eq:cocycle}, it follows that there exists a map $ \zeta:  \overline{F}  \times  \overline{F}  \to T_{ \overline{F} }$ such that for all $f_1, f_2 \in  \overline{F} $ we have 
\[  \alpha( f_1f_2)= \alpha(f_1)^{ f_2^{-1}} \alpha(f_2) c(f_1, f_2)  \zeta( f_1, f_2). \]
Now consider the modified section $ \widetilde{ s} :  \overline{F}  \to T$ defined by $ \widetilde{ s}(f)= s(f) \alpha(f)$. A simple computation shows that 
\[ \widetilde{ s}( f_1f_2) = \widetilde{ s}(f_1) \widetilde{ s}(f_2) \zeta( f_1, f_2).\]
It follows that $ F:= \widetilde{ s}( \overline{F} ) T_{ \overline{F} }$ is a finite subset which is closed under multiplication, and hence a subgroup of $G$. Clearly, we have $N_G(T)= TF$. As $T\subseteq G^\circ$, $G=N_G(T) G^\circ$, and $N_G(T)= TF$, the claim follows.
\end{proof}

\begin{lem}\label{lem:algebraic-gluing}
Let $G_0$ be a compact connected simple Lie group of adjoint type. 
Suppose $G$ is a closed subgroup of $\prod_{i=1}^n G_0$ so that $\pr_i(G)=G_0$ for all $i$, 
where $\pr_i$ denotes the projection to the $i$-th component; 
and assume that for all $i\neq j$ and all $\theta\in \Aut(G_0)$ there exists $x\in G$ 
such that $\theta\circ \pr_i(x)\neq \pr_j(x)$. Then $G=\prod_{i=1}^n G_0$.   	
\end{lem}

\begin{proof}
We proceed by induction on $n$. The base of induction is clear.  For any $i$, let $j_i:G_0\rightarrow \prod_{i=1}^{n} G_0$ be the natural injection to the $i$-th component and $N_i:=G\cap j_i(G_0)$. 

Suppose contrary to the claim that $N_i\neq j_i(G_0)$ for some $i$. 
Without loss of generality we can and will assume that $N_1\neq j_1(G_0)$. 
This and our assumption on $G$ imply that $N_1$ is a proper normal subgroup of $j_1(G_0)$, hence, it is trivial. 

We conclude that the projection $\pr_{[2..n]}: G\rightarrow \prod_{i=2}^n G_0$ to the components $2,\ldots,n$ is injective. 
Let $H_1:=\pr_{[2..n]}(G)\subseteq \prod_{i=2}^n G_0$. Clearly $H_1$ satisfies the same properties as $G$, hence by the inductive hypothesis, we have $H_1=\prod_{i=2}^n G_0$. 

Let $\xi:\prod_{i=2}^n G_0\rightarrow G$ be the inverse of the isomorphism $\pr_{[2..n]}:G\rightarrow H_1$. 
Let 
\[
\phi=\pr_1\circ\xi:\textstyle\prod_{i=2}^n G_0\rightarrow G_0.
\] 
Then $G$ can be identified with the graph of $\phi$, and $\phi$ is onto. 
Since $\ker(\phi)$ is a normal subgroup of $\prod_{i=2}^n G_0$, it follows that 
$$[\ker(\phi),j_i(G_0)] =j_i([\pr_i(\ker(\phi)),G_0]) \subseteq \ker (\phi).$$ 
As $G_0=[G_0,G_0]$, we have $\ker(\phi)=\textstyle\prod_{i\in I} G_0$ for some $I\subseteq [2..n]$. 

This shows that $\prod_{i\in [2..n]\setminus I} G_0\simeq G_0$, that is, $I=[2..n]\setminus\{i_1\}$ for some $i_1$, 
and $$\phi|_{j_{i_1}(G_0)}:j_1(G_0)\rightarrow G_0$$ is an isomorphism. Hence, there is $\theta\in \Aut(G_0)$ such that $\theta(\pi_1(x))=\pr_{i_1}(x)$ for any $x\in G$, which contradicts our assumption.  
\end{proof}

\begin{lem}\label{lem:proving-finite-number-reps}
Let $G$ be a compact group and assume that $d_{G,f}$ induces the original topology of $G$ for some strictly increasing function $f:\bbz^+\rightarrow \bbr^+$. 
Then $G$ has only finitely many pairwise non-isomorphic irreducible representations of dimension $n$ for any positive integer $n$. 
\end{lem}
\begin{proof}
Suppose contrary to the claim that there are infinitely many pairwise non-isomorphic 
$n$-dimensional irreducible representations $\{\pi_i\}_{i=1}^{\infty}\subseteq \wh{G}$. 

Let $G_i:= \pi_i(G)\subseteq \U_n(\bbc)$. By Lemma \ref{lem:Platonov}, for every  $i$ there exists a finite subgroup $F_i \subseteq  G_i$ such that $G_i= F_i G_i^{\circ}$.  By Jordan's theorem, $F_i$ contains a normal abelian subgroup $A_i$ of bounded index, depending on $n$. Let $ N_i=  \pi_i^{-1}( A_i G_i^\circ)$. Then $N_i$ is a normal open subgroup of $G$ 
with $ [G: N_i] \ll_n 1$. By Lemma~\ref{lem:SubgroupGrowth}, passing to a subsequence, we can and will assume that $N_i$ is a fixed subgroup $N$ for all $i \ge 1$. In particular, it follows that $N$ surjects onto all $A_i/(A_i\cap G_i^\circ)$ for $i \ge 1$. By Lemma~\ref{lem:TheFAb}, $N$ has a finite abelianization. Hence we must have $\sup_i [A_i:A_i\cap G_i^\circ] < \infty$, which implies that 
\be\label{eq:uniform-bound-index}
\sup_i [G:\pi_i^{-1}(G_i^\circ)]\le [G:N]\sup_i [A_i:A_i\cap G_i^{\circ}] <\infty.
\ee
As $H_i:=\pi_i^{-1}( G_i^\circ)$ is an open subgroup for all $i$, by the assumption and \eqref{eq:uniform-bound-index}, after passing to a subsequence, we may and will assume that $H_i$ is the same subgroup $H$ for all $i \ge 1$. 

Altogether, we have proved that there is an open subgroup $H$ of $G$ such that $\pi_i(H)$ are connected subgroups of $\U_n(\bbc)$ for all $i$. Since $H$ has a finite abelianization, $\pi_i(H)$ are semisimple connected subgroups of $\U_n(\bbc)$. 
There are only finitely many such subgroups, up to isomorphism. Hence, after passing to factors of $\pi_i(H)$ and a subsequence, we can and will assume that there is a compact connected simple Lie group of adjoint type $G_0$ such that $\pi_i(H)\simeq G_0$ for any $i$. 

The order of the group $\Aut(G_0)/\Inn(G_0)$ is bounded by a function of $n$ and $\pi_i$'s are pairwise not $G$-conjugate; 
thus, after passing to a subsequence we can and will assume that for any $i\neq j$ and $\theta\in \Aut(G_0)$, there is $x\in H$ 
such that $\theta(\pi_i(x))\neq \pi_j(x)$. 

By Lemma~\ref{lem:algebraic-gluing}, for any positive integer $m$,
\be\label{eq:finite-factors-onto}
\pi_{[1..m]}:H\rightarrow \prod_{i=1}^m G_0,\quad\quad \pi_{[1..m]}(x):=(\pi_i(x))_{i=1}^m
\ee   
is an onto group homomorphism. Now let us consider the group homomorphism
\[
\pi:H\rightarrow \prod_{i=1}^{\infty} G_0,\quad\quad \pi(x):=(\pi_i(x))_{i=1}^{\infty}.
\]
For $x_0\in G_0\setminus\{1\}$, by \eqref{eq:finite-factors-onto}, we get a sequence $\{g_m\}_{m=1}^{\infty}$ of elements of $G$ such that 
\be\label{eq:image-under-pi}
\pi(g_m)\in \prod_{i=1}^m\{1\}\times \{x_0\} \times \prod_{i=m+2}^{\infty} G_0;
\ee
and so $\lim_{m\rightarrow \infty} \pi(g_m)=1$. Since $\ker \pi$ is a compact subgroup, we can choose $g_m$'s in a way that $\lim_{m\rightarrow \infty} g_m=1$ and \eqref{eq:image-under-pi} holds. Let $\phi$ be a non-trivial irreducible representation of $G_0$, and let $\phi_i:=\phi\circ \pr_i\circ \pi\in \wh{G}$, where $\pr_i$ is the projection to the $i$-th component. Then by \eqref{eq:image-under-pi} we have 
\[
\phi_{m+1}(g_m)=\phi(x_0);
\] 
and so 
\[
d_{G,f}(g_m,1)\ge \frac{\|\phi_{m+1}(g_m)-I\|_{\op}}{f(\dim \phi_{m+1})}=\frac{\|\phi(x_0)-I\|_{\op}}{f(\dim \phi)}>0.
\]
On the other hand, by Lemma~\ref{lem:towards-def} and $\lim_{m\rightarrow \infty} g_m=1$, we have
\[
\lim_{m\rightarrow \infty} d_{G,f}(g_{m},1)=0,
\] 
which is a contradiction. 
\end{proof}

\begin{proof}[Proof of Theorem~\ref{thm:general-finite-number-irreducible-rep}]
Clearly (1) implies (2). Lemma~\ref{lem:proving-finite-number-reps} proves that (2) implies (3). Next we want to show that (3) implies (1). Suppose $f:\bbz^+\rightarrow \bbr^+$ is a strictly increasing function. By Lemma~\ref{lem:towards-def} we need to show that 
$\lim_{x\rightarrow 1} d_{G,f}(x,1)=0$. For a given $\vare>0$, there are only finitely many representations $\{\pi_1,\ldots,\pi_n\}\subset \wh{G}$ such that $f(\dim \pi_i)<2/\vare$. Hence, 
for all $\pi\in \wh{G} \setminus \{\pi_1,\ldots,\pi_n\}$ and all $x\in G$ we have 
\[
\frac{\|\pi(x)-I\|_{\op}}{f(\dim \pi)}\le \frac{2}{f(\dim \pi)}\le \vare.
\]
Since $\pi_i$'s are continuous and $G$ is compact, $\pi_i$'s are uniformly continuous. And so there is $\eta>0$ such that for all $x\in 1_{\eta}$ we have 
\[
\|\pi_i(x)-I\|_{\op}\le \vare f(1)
\]
for all $i\in [1..n]$. Altogether we get that for all $x\in 1_{\eta}$ and all $\pi\in \wh{G}$ we have 
\[
\frac{\|\pi(x)-I\|_{\op}}{f(\dim \pi)}\le \vare,
\]
which implies that $d_{G,f}(x,1)\le \vare$ for all $x\in 1_{\eta}$; and the claim follows.
\end{proof}
\begin{proof}[Proof of Theorem~\ref{thm:finitely-many-representations}]
Suppose $G$ is locally random; that means $G$ has a metric such that for all $x\in G$ and $\pi\in \wh{G}$ we have
\[
\|\pi(x)-I\|_{\op}\le C_0 (\dim \pi)^L d(x,1).
\]
Let $f:\bbz^+\rightarrow \bbr^+, f(n):=C_0 n^L$; then $f$ is strictly increasing and $\lim_{x\rightarrow 1} d_{G,f}(x,1)=0$. Hence, for every $n \ge 1$, it follows from  Theorem~\ref{thm:general-finite-number-irreducible-rep}, that there are only finitely many elements of $\wh{G}$ of dimension at most $n$.

Conversely,  suppose that for all integers $n \ge 1$, there are only finitely many elements of $\wh{G}$ of dimension at most $n$. Set $f:\bbz^+\rightarrow \bbr^+, f(n):=n$. By Theorem~\ref{thm:general-finite-number-irreducible-rep}, $d_{G,f}$ induces 
the original topology of $G$, and with respect to this metric for all $x,y\in G$ and $\pi\in \wh{G}$ we have 
\[
\|\pi(x)-\pi(y)\|_{\op} \le (\dim\pi) d_{G,f}(x,y);
\]  
therefore, $G$ is locally random.	
\end{proof}


\section{Local randomness, dimension condition, and important examples.}\label{sec:properties}
As we pointed out earlier, local randomness is particularly powerful when in addition 
the chosen metric has a \emph{dimension condition},~\eqref{eq:dimension-condition-intro}. 
Furthermore, several important examples, e.g., analytic compact groups, come equipped with a 
natural metric and we would like to know whether $G$ is locally random with respect to this natural metric. 

In this section we address this question. In particular, we show that compact simple Lie groups (with respect to their natural metric) are locally random; we also provide a connection between quasi-randomness and local randomness for profinite groups.  

We begin with investigating local randomness of quotients and products. 
Indeed, Theorem~\ref{thm:finitely-many-representations} implies that 
\begin{enumerate}
\item if $G$ is locally random and $N$ is a closed normal subgroup, then $G/N$ is locally random;
\item if $G_1$ and $G_2$ are locally random, then $G_1\times G_2$ is locally random.
\end{enumerate}
These statements, however, do not provide information regarding the metrics (or the involved parameters) with respect to which these groups are locally random. The following two lemmas prove the above statements with some control on the involved metrics.

  \begin{lem}
  Suppose $G$ is $L$-locally random  with coefficient $C_0$, and let 
  $N$ be a closed normal subgroup of $G$. Then $G/N$ equipped with the natural quotient metric is 
  $L$-locally random with coefficient $C_0$. 	
  \end{lem}
  
 \begin{proof}
Let us recall that given a bi-invariant metric $d$ on $G$, 
the natural quotient metric on $G/N$ is $d(xN,yN):=\inf_{h,h'\in N} d(xh,yh')$. 

For $\overline{\pi}\in \wh{G/N}$, let $\pi(x):=\overline{\pi}(xN)$; then $\pi\in \wh{G}$. For $x,y\in G$, and every $\vare>0$, there exist $h,h'\in N$ such that 
\[
d(xh,xh')< d(xN,yN)+\vare.
\] 
From this we conclude
\begin{align*}
\|\overline{\pi}(xN)-\overline{\pi}(yN)\|_{\op}&=\|\pi(xh)-\pi(yh')\|_{\op}\le C_0 (\dim\pi)^L d(xh,yh')\\
&\le C_0 (\dim \overline{\pi})^L (d(xN,yN)+\vare).
\end{align*} 	
The claim follows as $\vare$ is an arbitrary positive number.
  \end{proof}
  
  \begin{lem}
  Suppose $G_i$ is an $L_i$-locally random group with coefficient $C_i$ for $i=1,2$. Then $G_1\times G_2$ is an $\max\{L_1,L_2\}$-locally random group with coefficient $C_1+C_2$ with respect to the maximum metric. 	
  \end{lem}
  
  \begin{proof}
  We know that any $\pi\in \wh{G_1\times G_2}$ is of the form $\pi_1\otimes \pi_2$ for some $\pi_i\in \wh{G}_i$. It is also well-known that for any two matrices $a$ and $b$ we have $\|a\otimes b\|_{\op}=\|a\|_{\op}\|b\|_{\op}$. Let $L:=\max\{L_1,L_2\}$ and $C_0:=C_1+C_2$. Then for any $(g_1,g_2)\in G_1 \times G_2$ we have 
  \[
\begin{split}
	 \|\pi(g_1,g_2)-I\|_{\op}= & \|\pi_1(g_1)\otimes \pi_2(g_2)-I\otimes I\|_{\op} \\
	 \le & \|\pi_1(g_1)\otimes \pi_2(g_2)-I\otimes \pi_2(g_2)\|_{\op}+ \|I\otimes \pi_2(g_2)-I\otimes I\|_{\op}
	 \\
	 =
	 &
	 \|(\pi_1(g_1)-I)\otimes \pi_2(g_2)\|_{\op}+\|I\otimes (\pi_2(g_2)-I)\|_{\op}
	 \\
	 =
	 &
	 \|\pi_1(g_1)-I\|_{\op}+\|\pi_2(g_2)-I\|_{\op}
	 \\
	 \le 
	 & 
	 C_1 (\dim \pi_1)^L d_1(g_1,1)+C_2 (\dim \pi_2)^L d_2(g_2,1)\\
	 \le 
	 &
	 C_0 (\dim \pi)^L d((g_1,g_2),(1,1)),
\end{split}  
  \]
from which the claim follows.
   \end{proof}
   
   The following is essentially proved in \cite[Lemme 3.1, 3.2]{Saxce13}.
   
\begin{proposition}
Suppose $G$ is a compact simple Lie group. Then $G$ is $1$-locally random with coefficient $C_0:=C_0(G)$ with respect to the natural metric of $G$. 	
\end{proposition}

\begin{proof}
We briefly go over the proof given in \cite{Saxce13}. 
The Killing from is a negative definite bilinear form on the Lie algebra $\gfr$, therefore, 
\[
\langle X,Y\rangle:=-\tr(\ad(X)\ad(Y)) 
\] 
defines a bi-invariant inner product on $\gfr$ and hence a bi-invariant metric on $G$, which will be referred to as the natural metric of $G$.

Fix a maximal torus $T$ of $G$. Let $\Phi$ be the set of roots with respect to $T$ and let $\Phi^+$ be a set of positive roots.
Let $\pi$ be an irreducible unitary representation of $G$. Let $W_{\pi}:=\{\lambda_1,\ldots,\lambda_n\}$ be the set of weights 
of $\pi$, and let $\lambda$ denote the highest weight of $\pi$ with respect to $\Phi^+$. We have 
\[
\cal_\pi=\bigoplus_{j=1}^n \ker(\pi(\exp_T(X))-e^{i\lambda_j(X)}I).
\]
where $\exp_T$ denotes the restriction of the exponential map $\exp_G$ to $\mathfrak t$.

We also note that there is $\eta_0':=\eta_0'(G)$ such that for any $X\in \tfr$ 
with $\|X\|\leq \eta'_0$ we have 
\[
\|X\|\ll d(\exp_T(X),1)\ll \|X\|;
\]
Let $\eta_0=\eta_0(G)$ be chosen so that $1_{\eta_0}\subset \bigcup_{g \in G} g^{-1} \exp_T\bigl(\{X\in\mathfrak t: \|X\|\leq \eta'_0\}\bigr) g$.

Let $g\in 1_{\eta_0}$. Then $g$ is a conjugate of an element of the form $\exp_T(X)$ where $X\in\mathfrak t$
and $d(g,1)=d(\exp_T(X),1)$ ---recall that $d$ is $G$ bi-invaraint.
Hence, 
\be\label{eq:control-operator-norm}
\begin{split}
\|\pi_{\lambda}(g)-I\|_{\op}=&\|\pi_{\lambda}(\exp_T(X))-I\|_{\op}=\max_{\lambda_j\in W_{\pi_{\lambda}}} |e^{i\lambda_j(X)}-1|\\
\le & \max_{\lambda_j\in W_{\pi_{\lambda}}} |\lambda_j(X)|\ll \|\lambda\|\|X\|\ll \|\lambda\| d(\exp_T(X),1)=\|\lambda\| d(g,1).
\end{split}
\ee
On the other hand, by Weyl's formula 
\[
\dim \pi = \prod_{\alpha\in \Phi^+} \frac{\langle \lambda+\rho, \alpha\rangle}{\langle \rho,\alpha \rangle}
\]
where $\rho$ is the half of the sum of the positive roots. For every $ \alpha \in \Phi^+$, 
we have $ \frac{\langle \lambda+\rho, \alpha\rangle}{\langle \rho,\alpha \rangle} \ge 1$. 
Moreover, since the angle between every pair of distinct positive roots is more than $ \pi/2$, it follows that there exists $ \alpha \in \Phi^+$ for which  $\langle \lambda,\alpha\rangle\gg_G \|\lambda\|$. This implies that
\be\label{eq:control-dim}
\dim \pi \gg_G \|\lambda\|.  
\ee
By \eqref{eq:control-operator-norm} and \eqref{eq:control-dim} we get
\[
\|\pi_{\lambda}(g)-I\|_{\op} \le C_0'(G) (\dim \pi) d(g,1),
\]
for some $C_0'(G)$ and any $g\in 1_{\eta_0}$. Therefore, $G$ is $1$-locally random with 
coefficient $C_0 : = \frac{2 C_0'(G)}{\eta_0} $.
\end{proof}

\medskip

We now turn to the case of profinite groups. 
Following Varj\'{u}~\cite{Varju13}, a profinite group $G$ will be called $(c,\alpha)${\em-quasi-random} if for all $\pi\in \wh{G}$ we have 
\[
\dim \pi\ge c \h (\# \pi(G))^\alpha.
\]
This is a natural extension of Gowers's notion of quasi-randomness to profinite setting.

Our next objective in this section is to relate this notion, which does not depend on the metric structure of $G$, to local randomness. 
Indeed, if $G$ is a finitely generated $(c,\alpha)$-quasi-random group, then it has only finitely many irreducible unitary representations of a given dimension. Therefore, by Theorem~\ref{thm:finitely-many-representations}, we deduce that such a group is locally random. We will investigate this relationship in more details.

The following discussion is inspired by the $p$-adic setting. 
Suppose $G$ is equipped with a bi-invariant metric, and define the {\em level} of $\pi\in\wh{G}$ as:
\[
\ell(\pi):=\inf\{\eta^{-1}|\h 1_{\eta}\not\subseteq \ker \pi\},
\]
so for all $\vare>0$ we have $1_{(\ell(\pi)+\vare)^{-1}}\subseteq \ker \pi$. If $1_{\eta}$ is a normal subgroup for every $\eta>0$, then $\pi(G)$ is a factor of $G/1_{(\ell(\pi)+\vare)^{-1}}$. Hence
\be\label{eq:image-rep}
\# \pi(G) \le |1_{(\ell(\pi)+\vare)^{-1}}|^{-1}.
\ee
If, in addition, $G$ satisfies~\eqref{eq:dimension-condition-intro}, then we conclude from \eqref{eq:image-rep} that 
$
\# \pi(G)\le C_1 (\ell(\pi)+\vare)^{d_0}
$
for all $\vare>0$. Therefore, 
\be\label{eq:image-rep-2}
\# \pi(G)\le C_1 \ell(\pi)^{d_0}.
\ee

In view of the above inequality, we define a {\em metric quasi-randomness} for profinite groups.
\begin{definition}
A compact group $G$ with a given bi-invariant metric is said to be $(C,A)$-{\em metric quasi-random} if the following two conditions are satisfied: 
\begin{enumerate}
\item For all $\eta>0$, $1_{\eta}$ is a subgroup of $G$.
\item For all $\pi\in \wh{G}$, we have $\ell(\pi)\le C (\dim \pi)^A$. 	
\end{enumerate}	
\end{definition}
Hence by \eqref{eq:image-rep-2} we get the following:
\begin{lem}\label{lem:metric-quasi-random}
Suppose $G$ is an $(C,A)$-metric quasi-random group and 
$|1_{\eta}|\le C_1 \eta^{d_0}$ for all $\eta>0$ where $C_1$ and $d_0$ are positive constants. Then $G$ is $((C_1C^{d_0})^{-1}, 1/(Ad_0))$-quasi-random.
\end{lem}
Next we prove that $L$-local randomness (with some parameters) and metric quasi-randomness are equivalent when balls centered at $1$ are subgroups.  
\begin{proposition}\label{prop:loc-rand-metric-quasi-rand}
Suppose $G$ is a compact group with a bi-invariant metric. Suppose $G=1_1$, and $1_{\eta}$ is a subgroup of $G$ for all $\eta\in (0,1]$. Then $G$ is $L$-locally random with coefficient $C_0$ if and only if $G$ is $(C,L)$-metric quasi-random, where $C=C_0$ in one direction, and $C_0=2C$ in the other direction.
\end{proposition}
\begin{proof}
 Suppose $G$ is locally random, and let $\pi\in\wh{G}$ be non-trivial. For $x\in 1_{\eta}$ we have 
	\[
	\|\pi(x)-I\|_{\op}\le C_0 (\dim \pi)^L \eta.
	\]
	In particular, if $\eta<C_0^{-1} (\dim \pi)^{-L}$ and $x\in 1_{\eta}$, then for any $n\in \bbz$, $\|\pi(x)^n-I\|_{\op}<1$. 
	This implies $\log(\pi(x)^n)$ is well-defined for all integer $n$ ---recall that $\pi(x)\in U_{\dim\pi}(\mathbb C)$. 
	Furthermore, $\log(\pi(x)^n)=n\log(\pi(x))$. Since $G$ is profinite, $\pi(G)$ is a finite group, and hence, 
	$\pi(x)$ is torsion for any $x\in G$. Therefore, for some positive integer $n$ we have  $0=\log(\pi(x)^n)=n\log(\pi(x))$, which implies that $\pi(x)=I$. That is:
\be\label{eq:large-ball-in-kernel}
1_{\eta}\subseteq \ker \pi \h\h\h\text{ if } \eta<C_0^{-1} (\dim \pi)^{-L}.
\ee
By \eqref{eq:large-ball-in-kernel} we have 
\[
\ell(\pi)\le C_0 (\dim \pi)^L,
\]
which implies that $G$ is $(C_0,L)$-metric quasi-random. 

To see the other implication, note that for all $\pi\in \wh{G}$ and any $x\in G$, $\pi(x)\neq I$ implies that $d(x,1)\ge 1/\ell(\pi)$. 
Therefore, 
\[
\|\pi(x)-I\|_{\op}\le 2 \le 2 \ell(\pi) d(x,1)\le 2C(\dim \pi)^L d(x,1),
\]
which implies that $G$ is $L$-locally random with coefficient $2C$.
\end{proof}
In \cite[Lemma 20]{SG-JEMS}
using Howe's Kirillov theory, it is proved that an open compact subgroup $G$ of a $p$-adic analytic group with a perfect Lie algebra is $(C,A)$-metric quasi-random for some positive numbers $C$ and $A$ depending on $G$. Thus, by Proposition~\ref{prop:loc-rand-metric-quasi-rand} we obtain an important family of locally random groups.
\begin{proposition}
Suppose $G$ is a compact open subgroup of a $p$-adic analytic group with a perfect Lie algebra. Then, for some positive number $L$ and $C_0$, $G$ is $L$-locally random with coefficient $C_0$.
\end{proposition}


\section{Mixing inequality for locally random groups}\label{s:mixing-inq}
In this section, we will prove Theorem~\ref{thm:mixing-inequality} and derive a number of its corollaries. 
\subsection{High and low frequencies and the proof of Theorem~\ref{thm:mixing-inequality} }
The proof of Theorem~\ref{thm:mixing-inequality} involves splitting the terms in Parseval's theorem for $\|f \|^2$ into the sum of contributions from {\em low frequency} and {\em high frequency} terms. By low (resp.\ high) frequency terms, we mean terms coming from irreducible representations of small (resp.\ large) degree. The low frequency terms can be bounded by the local randomness assumption whereas high frequency terms are dealt with using a trivial bound. For  $f \in L^2(G)$ and a threshold parameter $D$, write
\be\label{eq:low-freq-terms}
L(f;D):= \sum_{\pi\in \widehat{G}, \dim \pi \le D} \dim \pi \h \|\widehat{f}(\pi)\|_{\HS}^2
\ee
for the low-frequency terms and 
\be\label{eq:high-freq-terms}
H(f;D):= \sum_{\pi\in \widehat{G}, \dim \pi > D} \dim \pi \h \|\widehat{f}(\pi)\|_{\HS}^2
\ee
for the high frequency terms. By Parseval's theorem, $\|f\|_2^2=L(f;D)+H(f;D)$ holds.

\begin{lem}\label{lem:low-and-high-covolution}
In the above setting, we have
\[
L(f\ast g;D)\le L(f;D) L(g;D)\quad \text{ and }\quad H(f\ast g; D)<\frac{1}{D} H(f;D) H(g;D).
\]	
\end{lem}
\begin{proof}
We have $\|\widehat{f\ast g}(\pi)\|_{\HS}=\|\widehat{g}(\pi)\widehat{f}(\pi)\|_{\HS}\le 	\|\widehat{g}(\pi)\|_{\HS}\|\widehat{f}(\pi)\|_{\HS}$, and 
\begin{align*}
	L(f\ast g;D)=& \sum_{\pi\in \widehat{G}, \dim \pi\le D} \dim \pi  \|\widehat{f\ast g}(\pi)\|_{\HS}^2 \\
	\le & \left(\sum_{\pi\in \widehat{G}, \dim \pi\le D} \dim \pi  \|\widehat{f}(\pi)\|_{\HS}^2\right)\left(\sum_{\pi\in \widehat{G}, \dim \pi\le D} \dim \pi  \|\widehat{g}(\pi)\|_{\HS}^2\right)\\
	\le & L(f;D) L(g;D).
\end{align*}
Similarly, we have the inequality
\begin{align*}
	H(f\ast g;D)=& \sum_{\pi\in \widehat{G}, \dim \pi> D} \dim \pi  \|\widehat{f\ast g}(\pi)\|_{\HS}^2 \\
	< & \frac{1}{D}\left(\sum_{\pi\in \widehat{G}, \dim \pi> D} \dim \pi  \|\widehat{f}(\pi)\|_{\HS}^2\right)\left(\sum_{\pi\in \widehat{G}, \dim \pi> D} \dim \pi  \|\widehat{g}(\pi)\|_{\HS}^2\right)\\
	\le & \frac{1}{D}H(f;D) H(g;D),
\end{align*}
as we claimed.
\end{proof}

\begin{lem}[Fourier terms in low frequencies]\label{lem:low-freq}
	Suppose $G$ is an $L$-locally random with coefficient $C_0$. Then for all $\eta>0$ and $\pi\in \widehat{G}$ we have
	\[
	\|\widehat{P}_{\eta}(\pi)-I\|_{\op}\le C_0 (\dim \pi)^L\h \eta.
	\]
\end{lem}
\begin{proof}
For all $x\in 1_{\eta}$, we have $\|\pi(x)-I\|_{\op}\le C_0 (\dim \pi)^L d(1,x)\le C_0 (\dim \pi)^L\h \eta$. Therefore,
\[
\|\widehat{P}_{\eta}(\pi)-I\|_{\op}=\biggl\|\int P_{\eta}(x) (\pi(x)^{\ast}-I) \d x\biggr\|_{\op}\le  C_0 (\dim \pi)^L\h \eta.
\]
\end{proof}

\begin{lem}\label{lem:norm-eta-averaging}
Suppose $G$ is an $L$-locally random group with coefficient $C_0$. Let $\eta>0$ and $D\geq1$ be two parameters 
satisfying $C_0D^L\eta<1$. Then
\[
L(f;D)\le (1-C_0D^L \eta)^{-2} L(f_{\eta};D)\le (1-C_0D^L \eta)^{-2} \|f_{\eta}\|_2^2
\]
where $f_\eta=P_\eta*f$.
\end{lem}

\begin{proof}
The second inequality is clear because of $L(f_{\eta};D)\leq \|f_{\eta}\|_2^2$. 

We now show the first inequality. Note that
\be\label{eq:L-feta}
L(f_{\eta};D)= \sum_{\pi\in \widehat{G}, \dim \pi \le D} \dim \pi \h \|\widehat{P}_{\eta}(\pi)\widehat{f}(\pi)\|_{\HS}^2.
\ee
We have
\begin{align}
	\notag
	\|\widehat{P}_{\eta}(\pi)\widehat{f}(\pi)\|_{\HS}=&\|\widehat{f}(\pi)-(I-\widehat{P}_{\eta}(\pi))\widehat{f}(\pi)\|_{\HS} \\
\notag	\ge & (1-C_0 (\dim \pi)^L\h \eta) \|\widehat{f}(\pi)\|_{\HS} &\text{(By Lemma~\ref{lem:low-freq})}\\
\notag\ge & (1-C_0 D^L\h \eta) \|\widehat{f}(\pi)\|_{\HS}
\end{align}
This estimate and~\eqref{eq:L-feta} imply that
\be\label{eq:final-lower-bound-low-freq}
L(f_{\eta};D)\ge (1-C_0D^L\eta)^2 L(f;D)
\ee
which finishes the proof.
\end{proof}

\begin{proof}[Proof of Theorem~\ref{thm:mixing-inequality}]
	
	Let $\eta$ be as in the statement of the theorem, and let $D=(\sqrt{\eta})^{-1/L}$. 
	
	By Lemmas \ref{lem:low-and-high-covolution} and \ref{lem:norm-eta-averaging}, we have
\begin{align*}
	\|f\ast g\|_2^2= & L(f\ast g;D)+H(f\ast g;D) \\
	\le & L(f;D)L(g;D)+\frac{1}{D} H(f;D)H(g;D) \\
	\le & (1-C_0D^L\eta)^{-4} \|f_{\eta}\|_2^2 \|g_{\eta}\|_2^2+ \frac{1}{D}\|f\|_2^2\|g\|_2^2.
\end{align*}
Note that $(1-C_0D^L\eta)^{-4}=(1-C_0\sqrt\eta)^{-4}\leq 0.9^{-4}\leq 2$. The claim follows from here. 
\end{proof}


\subsection{An almost orthogonality and further mixing inequalities}\label{s:almost-orthog}
The inequality in Theorem~\ref{thm:mixing-inequality} is non-trivial only when $\|f_{\eta}\|_2$ and $\|g_{\eta}\|_2$ are small. 
In this section, we show that $(f-f_{\eta})_{\eta'}$ is small when $\eta$ is polynomially smaller than $\eta'$. 
Thus applying the mixing inequality of Theorem~\ref{thm:mixing-inequality} to $(f-f_{\eta})_{\eta'}$ and $g$, we get a meaningful mixing. We will then use this to  prove a product theorem. To get a better understanding of the discussion, 
consider the case when $1_{\eta}$ is a subgroup of $G$. 
Then $f\mapsto f_{\eta}$ is the orthogonal projection onto the space of $1_{\eta}$-invariant functions in $L^2(G)$ and $(f-f_{\eta})_{\eta}=0$; hence, one may let $\eta'=\eta$.

Results in this section require only a dimension condition at a given scale. 
This is implied by~\eqref{eq:dimension-condition-intro}, but is more general.



Let us recall that any class function in $L^1(G)$ is in the center of the Banach algebra $(L^1(G),+,\ast)$; therefore, $P_{\eta}$ is in the center of $L^1(G)$ for any $\eta$. 
 
\begin{lem}\label{lem:almost-orthogonality}
Suppose $G$ is an $L$-locally random group with coefficient $C_0$. For every $C_1>0$ and every $\eta\ll_{C_0,C_1,L} 1$ we have the following. Suppose $\eta'\ge \eta^{1/(4Ld_0)}$ satisfies $|1_{\eta'}|\ge \frac{1}{C_1} \eta'^{d_0}$. Then for every $f\in L^2(G)$ we have 
\[
\|(f-f_{\eta})_{\eta'}\|_2\le \eta^{1/(8L)}\|f\|_2.
\]
\end{lem}

\begin{proof}
Let $D$ be a threshold parameter which will be set later. Then
\begin{align*}
L((f-f_{\eta})_{\eta'};D)= & \sum_{\pi\in \widehat{G},\dim \pi\le D} \dim \pi \|\widehat{f}(\pi)\widehat{P}_{\eta'}(\pi)(I-\widehat{P}_{\eta}(\pi))\|_{\HS}^2
\\
\le & \sum_{\pi\in \widehat{G},\dim \pi\le D} \dim \pi \|I-\widehat{P}_{\eta}(\pi)\|_{\op}^2\|\widehat{P}_{\eta'}(\pi)\|_{\op}^2\|\widehat{f}(\pi)\|_{\HS}^2 \\
 \le & (C_0 D^L \eta)^2 L(f;D) &&\text{(By Lemma~\ref{lem:low-freq})}.
\end{align*}
We used $\|AB\|_{\HS}\le \|A\|_{\op}\|B\|_{\HS}$ for matrices $A$ and $B$ for the first inequality, 
and $\|\widehat{P}_{\eta'}(\pi)\|_{\op}\le 1$ in the final inequality. For the high frequencies we have
\begin{align*}
H((f-f_{\eta})_{\eta'};D)= &\sum_{\pi\in \widehat{G},\dim \pi> D} \dim \pi \|\widehat{f}(\pi)\widehat{P}_{\eta'}(\pi)(I-\widehat{P}_{\eta}(\pi))\|_{\HS}^2 \\
\le & \sum_{\pi\in \widehat{G},\dim \pi> D} \dim \pi \|I-\widehat{P}_{\eta}(\pi)\|_{\op}^2\|\widehat{P}_{\eta'}(\pi)\|_{\op}^2\|\widehat{f}(\pi)\|_{\HS}^2 \\
\le & \frac{4}{D} H(P_{\eta'};D) H(f;D)\le \frac{4}{D} \|P_{\eta'}\|_2^2 H(f;D) = \frac{4}{D\h |1_{\eta'}|} H(f;D),
\end{align*}
where we used the trivial bound $\|I-\widehat{P}_{\eta}(\pi)\|_{\op}\leq 2$. Combining these two estimates, we conclude
\[
\|(f-f_{\eta})_{\eta'}\|_2^2\le \left((C_0 D^L \eta)^2+\frac{4}{D\h |1_{\eta'}|}\right) \|f\|_2^2.
\]
Setting $D=\eta^{-1/(2L)}$ we get the desired inequality.	
\end{proof}

In the rest of this section we will prove a number of mixing inequalities. 

\begin{lem}\label{lem:l2-mixing-loc-random}
	Suppose $G$ is an $L$-locally random group with coefficient $C_0$. For every integer $m\ge 2$, $C_1>0$, and $\eta\ll_{C_0,C_1,L} 1$ we have the following. Suppose $\eta'\ge \eta^{1/(4Ld_0)}$ satisfies $C_0\sqrt{\eta'}<0.1$ and $|1_{\eta'}|\ge \frac{1}{C_1} \eta'^{d_0}$. Then for all $f_1,\ldots,f_m\in L^2(G)$ we have
	\[
	\|(f_1-(f_1)_{\eta})\ast f_2 \ast \cdots \ast f_m\|_2 \le \sqrt{3}^{m}\eta'^{1/(4L)} \prod_{i=1}^m \|f_i\|_2.
	\]
\end{lem}

\begin{proof}
We proceed by induction on $m$. Let us start with the base case $m=2$. By Theorem~\ref{thm:mixing-inequality}, we have that $\|(f_1-(f_1)_{\eta})\ast f_2\|_2^2$ is bounded from above by 
\be\label{eq:initial-mixing-inequality-base-of-induction}
2 \|(f_1-(f_1)_{\eta})_{\eta'}\|_2^2 \|(f_2)_{\eta'}\|_2^2 + \eta'^{1/(2L)} \|f_1-(f_1)_{\eta}\|_2^2 \|f_2\|_2^2.
\ee

By Lemma~\ref{lem:almost-orthogonality} we have
\be\label{eq:almost-orhogonality}
\|(f_1-(f_1)_{\eta})_{\eta'}\|_2\le \eta^{1/(8L)} \|f_1\|_2.
\ee
Since $\|f\ast g\|_2\le \|f\|_1\|g\|_2$, we have
\be\label{eq:second-term}
\|f_1-(f_1)_{\eta}\|_2\le \|1-P_{\eta}\|_1 \|f_1\|_2\le 2\|f_1\|_2\quad
\text{ and }\quad
\|(f_2)_{\eta'}\|_2\le \|f_2\|_2.
\ee
By \eqref{eq:initial-mixing-inequality-base-of-induction}, \eqref{eq:almost-orhogonality}, 
and \eqref{eq:second-term}, we get that $\|(f_1-(f_1)_{\eta})\ast f_2\|_2^2$ is bounded from above by 
\[
2(\eta^{1/(8L)} \|f_1\|_2)^2 \|f_2\|_2^2+ 
4\eta'^{1/(2L)} \|f_1\|_2^2\|f_2\|_2^2.
\]
Therefore
\[
\|(f_1-(f_1)_{\eta})\ast f_2\|_2 \le  \sqrt{ 6} \eta'^{1/(4L)} \|f_1\|_2 \|f_2\|_2.
\]
This concludes the proof for $m=2$. Now, suppose that the inequality holds for some value of $m$, and set 
\[
F_m:=(f_1-(f_1)_{\eta})\ast f_2 \ast \cdots \ast f_m.
\]
By Theorem~\ref{thm:mixing-inequality}, we have that $\|F_m\ast f_{m+1}\|_2^2$ is at most
\[
2 \|(F_m)_{\eta'}\|_2^2\|(f_{m+1})_{\eta'}\|_2^2+
\eta'^{1/(2L)} \|F_m\|_2^2\|f_{m+1}\|_2^2.
\]
Since $\|(F_m)_{\eta'}\|_2\le \|F_m\|_2$, by the induction hypothesis we have
\[
\|(F_m)_{\eta'}\|_2^2\le \|F_m\|_2^2 \le 3^{m}\eta'^{1/(2L)} \prod_{i=1}^m\|f_i\|_2^2.
\]
Hence by the above inequalities we get $\|F_m\ast f_{m+1}\|_2^2$ is at most
\[
3^{m}(2+\eta'^{1/(2L)})\eta'^{1/(2L)}\prod_{i=1}^{m+1}\|f_i\|_2^2\le 3^{m+1} \eta'^{1/(2L)}\prod_{i=1}^{m+1}\|f_i\|_2^2;
\]
and the claim follows.
\end{proof}

\begin{proposition}\label{prop:pointwise-mixing-loc-random}
Suppose $G$ is an $L$-locally random group with coefficient $C_0$. For every integer $m\ge 2$, $C_1>0$, and $\eta\ll_{C_0,C_1,L} 1$ we have the following. Suppose $\eta'\ge \eta^{1/(4Ld_0)}$ satisfies $C_0\sqrt{\eta'}<0.1$ and $|1_{\eta'}|\ge \frac{1}{C_1} \eta'^{d_0}$. 
Suppose $f_1,\ldots,f_m,f_{m+1}\in L^2(G)$. Then
	\[
	\|f_1\ast \cdots \ast f_{m+1}-f_1\ast \cdots \ast f_{m+1}\ast P_{\eta}\|_{\infty} \le 
	\sqrt{3}^{m}\eta'^{1/(4L)} \prod_{i=1}^{m+1} \|f_i\|_2.
	\]
\end{proposition}

\begin{proof}
Recall that $\|f\ast g\|_{\infty}\le \|f\|_2\|g\|_2$, see~\eqref{eq:Young-ineq}. 
Therefore, from the fact that $P_{\eta}$ is in the center of $(L^1(G),+,\ast)$ we obtain
\[
\|f_1\ast \cdots \ast f_{m+1}-f_1\ast \cdots \ast f_{m+1}\ast P_{\eta}\|_{\infty}\leq \|(f_1-(f_1)_\eta)\ast \cdots \ast f_m\|_2\|f_{m+1}\|_{2}.
\]
The claim thus follows from Lemma~\ref{lem:l2-mixing-loc-random}.
\end{proof}

\begin{corollary}\label{cor:pointwise-mixing-loc-random}
Suppose $G$ is an $L$-locally random group with coefficient $C_0$. For every integer $m\ge 2$, $C_1>0$, and $\eta\ll_{C_0,C_1,L} 1$ we have the following. Suppose $\eta'\ge \eta^{1/(4Ld_0)}$ satisfies $C_0\sqrt{\eta'}<0.1$ and $|1_{\eta'}|\ge \frac{1}{C_1} \eta'^{d_0}$. 
Suppose $f_1,\ldots,f_m,f_{m+1}\in L^2(G)$. Then
\[
	\|f_1\ast \cdots \ast f_{m+1}-(f_1)_{\eta}\ast \cdots \ast (f_{m+1})_{\eta}\|_{\infty} \le m \sqrt{3}^{m}\eta'^{1/(4L)} \prod_{i=1}^{m+1} \|f_i\|_2.
	\]
\end{corollary}
\begin{proof}
Let $F_1:=f_1\ast\cdots \ast f_{m+1}$ and $F_{k+1}:=(f_1)_{\eta}\ast \cdots \ast (f_k)_{\eta} \ast f_{k+1} \ast \cdots \ast f_{m+1}$ for any $1\le k\le m$. By Proposition~\ref{prop:pointwise-mixing-loc-random} and the fact that $P_{\eta}$ is in the center of $(L^1(G),+,\ast)$ for any $k$ we have
\begin{align*}
	\|F_k-F_{k+1}\|_{\infty} \le & 
	\sqrt{3}^{m}\eta'^{1/(4L)} \prod_{i=1}^{k-1} \|(f_i)_{\eta}\|_2 \prod_{i=k}^{m+1} \|f_i\|_2
	\\
	 \le &
	  \sqrt{3}^{m}\eta'^{1/(4L)} \prod_{i=1}^{m+1} \|f_i\|_2.
\end{align*}
Therefore, $\|F_1-F_{m+1}\|_{\infty}\le m  \sqrt{3}^{m}\eta'^{1/(4L)} \prod_{i=1}^{m+1} \|f_i\|_2$, and the claim follows.	
\end{proof}


\newcommand{\high}{\mathrm{high}}
\section{A product result for large subsets.}\label{sec:product}
The main goal of this section is to prove Theorem~\ref{thm:product-large-subsets}.
We start by recalling a number of definitions and setting some notation. 
 Suppose $X$ is a metric space and $A$ is a non-empty subset of $X$. Recall that for $\eta\in (0,1)$, $x_{\eta}$ denotes the ball of radius $\eta$ centered at $x$, and similarly $A_{\eta}$ denotes the union of all $x_{\eta}$ with $x \in A$.   We write  $\ncal_{\eta}(A)$ for the least number of open balls of radius $\eta$ with centers in $A$ that cover $A$. The metric entropy of $A$ at scale $\eta$ is defined by 
 $h(A;\eta):=\log \ncal_{\eta}(A)$. A maximal $\eta$-separated subset $\ccal$ of $A$ has the property that every distinct $x,x'\in \ccal$ are at least $\eta$ apart and its $\eta$-neighborhood covers $A$.

The metric space we will be working with is a metrizable compact group $G$ equipped with bi-invariant metric denoted by $d(\cdot, \cdot )$. We will assume further that the pair $(G,d)$ enjoys the dimension condition $\Dim(C, d_0)$ defined in \eqref{eq:dimension-condition-intro}.


\begin{lem}[Uniformly comparable quantities]\label{lem:basic-properties-metric-entropy}
Fix a subset $A \subseteq X$ and $\eta>0$, and let $A^{\ast} \subseteq A $ be a maximal $\eta$-separated subset of $A$, and write 
$ \overline{A}= ( A^{\ast})_{\eta}$. Then $A^{\ast}$ is finite, $ \overline{A}$ is open, and $A^{\ast } \subseteq A  \subseteq \overline{A}$. Moreover, the ratio of any two quantities among  $$  | \overline{A}|/|1_\eta|, \quad |A_{\eta}|/|1_\eta|,  \quad \ncal_{\eta}(A), \quad  \# A^{\ast}$$ 
is bounded above by  $ \Omega= {2^{d_0}C^2}$.
\end{lem}

\begin{proof}
Write $N=\ncal_\eta(A)$ and denote by $\{(x_i)_{\eta}\}_{i=1}^{N}$ a minimal $\eta$-cover of $A$ with centers in $A$. For each $x\in A_{\eta}$ there exists some $1 \le i \le N$ such that $x\in (x_i)_{2\eta}$, implying that $A_{\eta}\subseteq \bigcup_{i=1}^N (x_i)_{2\eta}$. Therefore
\be\label{eq:lower-bound-metric-entropy}
|A_{\eta}|\le N |1_{2\eta}| \le {2^{d_0}C^2} N |1_{\eta}|.
\ee
where the last inequality follows from an application of \eqref{eq:dimension-condition-intro}. 

Since $ A^{\ast}$ is a maximal $ \eta$-separated subset of $A$, the open balls $\{x_{\eta}\}_{x\in A^{\ast}}$ form an $\eta$-cover of $A$ with centers in $A$, and hence 
\be\label{eq:entropy-separating}
\ncal_{\eta}(A)\le \# A^{\ast}.
\ee
Finally, since $A^{\ast}$ is $\eta$-separated, each two balls in the family $\{ x_{\eta/2}: x \in A^{\ast} \}$ are pairwise disjoint, yielding
$$|A^{\ast}_{\eta/2}| =  (\# A^{\ast})  \  |1_{\eta/2}|.$$ 
This implies that 
\be\label{eq:upper-bound-separating}
 \# A^{\ast} \le \frac{|A^{\ast}_{\eta}|}{|1_{\eta/2}|} \le {2^{d_0}C^2} \frac{|A_{\eta}|}{|1_{\eta}|}.
\ee
This completes the proof. 
 \end{proof}
 
\begin{remark}\label{wiggly}
From now on, whenever two positive quantities $X$ and $Y$ are within a multiplicative factor of the form $ \Omega^{O(1)}$ of one another, we will write $X\approx Y$. Similarly, we write $X \preccurlyeq Y$ to state that $X/Y$ is bounded from above by an expression of the form $ \Omega^{O(1)}$, where the implied constants are not of importance. Using this notation we can now write  
 
\[
\ncal_{\eta}(A)\approx \# A^{\ast} \approx \frac{|A_{\eta}|}{|1_{\eta}|}.
\]
\end{remark}

\begin{remark}\label{rem:boundedscale}
The proof of Lemma \ref{lem:basic-properties-metric-entropy} only uses the dimension condition for $\eta, 2\eta$ and $ \eta/2$. We will use this fact later. 
\end{remark}


\begin{corollary}\label{cor:larger-nbhd}
Suppose $G$ is a compact group that satisfies \eqref{eq:dimension-condition-intro}. Then for every fixed constant $c \ge 1$ and every non-empty subset $A$ of $G$ and every $0<\eta<1$ we have
\[
|A_{c\eta}|\approx |A_{\eta}|
\]
\end{corollary}
\begin{proof}
Since $|A_{c\eta}|\ge |A_{\eta}|$, we will need to prove the reverse inequality. Denote by $A^\ast(\eta)$ and $A^{\ast}( c \eta)$, respectively, maximal $\eta$-separated and $ c \eta$-separated subsets of $A$.  
By Lemma~\ref{lem:basic-properties-metric-entropy} we have that $\#A^{\ast}( c \eta) \approx \frac{|A_{c\eta}|}{|1_{c\eta}|}$. Clearly we have $	\# A^{\ast}( c \eta) \le \# A^{\ast}(  \eta)$, implying
\[
\frac{|A_{c\eta}|}{|1_{c\eta}|} \preccurlyeq \frac{|A_{\eta}|}{|1_{\eta}|}.
\]
Hence 
\[
|A_{c\eta}| \preccurlyeq \frac{|1_{c\eta}|}{|1_{\eta}|} |A_{\eta}| \approx  |A_{\eta}|;
\]
and the claim follows.
\end{proof}

For a Borel measurable set $A \subseteq G$ with $|A|>0$ and $ \eta>0$, define 
\[ \chi_{A, \eta}= \left( \frac{1}{|A|}\bb1_A \right) \ast 1_{\eta}.\]
Some basic properties of $\chi_{A, \eta}$ are summarized in the next lemma:

\begin{lem}\label{lem:bluring-normalized-char-functions}  
Let $G$ be as above and $0<\eta<1$. 
 \begin{enumerate}
\item  For a measurable subset of positive measure $A\subseteq G$, we have 
$$ \chi_{A, \eta}(x) =\frac{|A \cap x_{\eta}|}{|A||1_{\eta}|}.$$ 
\item $\chi_{A, \eta}$ is supported on $\eta$-neighborhood of $A$ and has  $L^{ \infty}$ norm at most $ 1/|A|$.
\item For  $A \subseteq B$ of positive measure
\[ \chi_{A, \eta}(x) \le \frac{|B|}{|A|} \chi_{B, \eta}(x). \]
\item 
If $d(x, y)< \rho<1$, then 
\[
\chi_{A, \eta}(x)\preccurlyeq \left(\frac{\eta+\rho}{\eta}\right)^{d_0} \chi_{A, \eta+\rho}(y).
\]	

\end{enumerate}

\end{lem}

\begin{proof}
Since $1_{\eta}$ is a symmetric subset, we have 
\[
\chi_{A, \eta}(x)=\frac{1}{|A||1_{\eta}|}\int_G \cf_{A\cap x_{\eta}}(y) \d y,
\]	
from which part (a) follows. Part (b) follows immediately from part (a). To show part (c), observe that $y_{\eta+\rho}\supseteq x_{\eta}$. It thus follows from 
the dimension condition that 
\begin{equation}
\chi_{A, \eta}(x)	\le   \frac{| A \cap y_{\eta+\rho}|}{|A||1_{\eta}|} 
	=  \frac{|1_{\eta+\rho}|}{|1_{\eta}|} \chi_{A, \eta+\rho}(y)\preccurlyeq \left(\frac{\eta+\rho}{\eta}\right)^{d_0}\chi_{A, \eta+\rho}(y).
\end{equation}
\end{proof}
The next lemma, which is a version of Markov's inequality, establishes another quantity that is comparable to the ones in 
Lemma \ref{lem:basic-properties-metric-entropy}.

\begin{lem}[Density points]\label{lem:MarkovInequality}
	Let $G$ be as above, $ A \subseteq G$ and $0<\eta<\rho<1$. Fixing $ \eta$, let $A^\ast$ be a maximal $\eta$-separated subset of $A$, and $\overline{A}:=A^\ast_{\eta}$. For  a threshold parameter $0<\tau<1$, we let
	\[
	A_\high:=\{x\in A^\ast : \h \chi_{\overline{A}, 3 \rho}(x)>\tau\}.
	\]  
	Under the condition that $\tau\preccurlyeq 1$, we have
	\[
	|\overline{A}|\preccurlyeq |(A_\high)_{\rho}|.
	\]
\end{lem}
\begin{proof}
	Every point $x$ in the support of $\chi_{\overline{A}, \rho}$ lies at distance less than $ \rho$ from $ \overline{A}$ and hence at distance less than $\eta + \rho < 2 \rho$ from a point $ \overline{x} \in A^{\ast}$:
	\[ \supp \chi_{ \overline{A},  \rho} \subseteq  (A^{\ast})_{2 \rho}. \]
	
	 By part (4) of Lemma~\ref{lem:bluring-normalized-char-functions} we have
	\be\label{eq:approx-in-discretized-subset}
	\chi_{\overline{A}, \rho} (x) \preccurlyeq  \chi_{\overline{A}, 3 \rho}(\overline{x}).
	\ee
	Write $Z=(A_\high)_{2\rho}$.  If $x \in G \setminus Z$, the above $\overline{x}$ is in $A^\ast \setminus A_\high$, which means
	\be\label{eq:upper-bound-value-disc}
	\chi_{\overline{A}, 3 \rho} (\overline{x})\le \tau.
	\ee
	By \eqref{eq:approx-in-discretized-subset} and \eqref{eq:upper-bound-value-disc} we deduce that for $x\in G \setminus Z$ and  $\tau \preccurlyeq 1$
	\[	
	\chi_{\overline{A}, \rho} (x)  \le 1/2.
	\]
This means that the density function  $ \chi_{\overline{A},  \rho}$  is concentrated on $Z$: 
\begin{equation}
1/2 \le \int_{Z} \chi_{\overline{A},  \rho}(x) \d x  
	\le   \frac{|(A_\high)_{2 \rho}|}{|\overline{A}|};
\end{equation}
where the last inequality follows from the fact that $\chi_{ \overline{A}, \rho}$ is bounded by $ 1/| \overline{A}|$. 
The claim  now follows from Corollary~\ref{cor:larger-nbhd}. 
\end{proof}
\begin{proof}[Proof of Theorem~\ref{thm:product-large-subsets}]	
As before we will choose maximal $ \eta$-separated subsets $A^\ast \subseteq A$ and  $B^\ast \subseteq B$, set  $ \overline{A}=
(A^\ast)_\eta$ and $ \overline{B}= (B^\ast)_\eta$.  Also write $C= B^{-1} {A}^{-1}$, and
$ \overline{C}= \overline{B}^{-1} \overline{A}^{-1}$. Note that in this proof we are deviating from the notation 
we used earlier in that here $ \overline{C}$ is {\it not} defined to be $(C^{\ast})_{\eta}$.

By the mixing inequality given in Proposition~\ref{prop:pointwise-mixing-loc-random} for $\rho:=\eta^{\vare}$ we have
	\be\label{eq:mixing-for-product}
	\|\chi_{\overline{A}}\ast \chi_{\overline{B}} \ast \chi_{\overline{C}}-\chi_{\overline{A},   5 \rho} \ast  \chi_{\overline{B}, 5 \rho}  \ast \chi_{\overline{C}, 5 \rho}  \|_{\infty}\le \rho^{O_{L,d_0}(1)} \|\chi_{\overline{A}}\|_2 \|\chi_{\overline{B}}\|_2 \|\chi_{\overline{C}}\|_2.
	\ee

The main step of the proof is to show the following inequality:
\[
	\chi_{\overline{A}, 5\rho}\ast \chi_{\overline{B}, 5\rho} \ast \chi_{\overline{C}, 5\rho}(x) 
	\succcurlyeq
 	\frac{(|\overline{A}||\overline{B}|)^{3/2}}{|\overline{C}|}.
\]	

Let $\tau\preccurlyeq 1$ be as in Lemma~\ref{lem:MarkovInequality}. For any $y\in (A_{\high})_{\rho}$ there is $y' \in A_\high$ such that $d(y',y) < \rho$. By part (4) of Lemma~\ref{lem:bluring-normalized-char-functions}, we have that 
\be\label{eq:A-factor-is-large}
\chi_{\overline{A},5 \rho}(y) \succcurlyeq  \chi_{\overline{A},4 \rho}(y) \succcurlyeq  \chi_{\overline{A}, 3 \rho}(y')\succcurlyeq 1.
\ee
Similarly for $z\in (B_\high)_{\rho}$ we have
\be\label{eq:B-factor-is-large}
\chi_{\overline{B}, 4 \rho}(z) \succcurlyeq 1.
\ee
For $y\in (A_{\high})_{\rho}$, $z\in (B_{\high})_{\rho}$, and $x\in 1_{\rho}$, 
by part (4) of Lemma~\ref{lem:bluring-normalized-char-functions} we have
\be\label{eq:C-factor-1}
\chi_{\overline{C}, {4\rho}  }(z^{-1}y^{-1}x) \succcurlyeq \chi_{\overline{C},3 \rho}(z^{-1}y^{-1}).
\ee
On the other hand, by part (1) of Lemma~\ref{lem:bluring-normalized-char-functions} we have
\be\label{eq:C-factor-2}
\chi_{\overline{C}, 3 \rho}(z^{-1}y^{-1})=
\chi_{\overline{C}^{-1}z^{-1}, 3 \rho}(y)=
\chi_{y^{-1} \overline{C}^{-1}, 3 \rho}(z).
\ee
Since $z\in (B_{\high})_{\rho}$, there exists some $z'\in B_{\high}$ so that $d(z,z')\leq\rho$. Moreover, using the definition 
$C=B^{-1}A^{-1}$, we have that $\overline{A}\subseteq \overline{C}^{-1}z'^{-1}$. Similarly,  from $d(y,y')\leq\rho$, we see that  $\overline{B}\subseteq y'\overline{C}^{-1}$. Hence by  \eqref{eq:C-factor-1}, and \eqref{eq:C-factor-2} and the estimate \eqref{eq:A-factor-is-large} we have
\begin{align*}
\chi_{\overline{C}, 5 \rho}(z^{-1}y^{-1}x)  & \succcurlyeq  \chi_{\overline{C}, 4 \rho}(z'^{-1}y^{-1}x) 
&& \textrm{ \tiny{ (part (4) of Lemma \ref{lem:bluring-normalized-char-functions}})} \\ 
&  \succcurlyeq  \chi_{\overline{C}, 3 \rho} (z'^{-1}y^{-1})  &&  \textrm{ \tiny{ (by \eqref{eq:C-factor-1}})} \\
&= \chi_{\overline{C}^{-1}z'^{-1}, 3 \rho}(y) &&  \textrm{ \tiny{ (by \eqref{eq:C-factor-2}})} \\
  & \succcurlyeq\frac{|\overline{A}|}{|\overline{C}|}  \chi_{\overline{A},3\rho}(y) && \textrm{ \tiny{ (part (3) of Lemma \ref{lem:bluring-normalized-char-functions}})}   \\
  & \succcurlyeq\frac{|\overline{A}|}{|\overline{C}|}.  &&  \textrm{ \tiny{ (by $y \in A_\high$})} \\
\end{align*}
Similarly, 
\[
\chi_{\overline{C}, 5\rho}(z^{-1}y^{-1}x)
\succcurlyeq\frac{|\overline{B}|}{|\overline{C}|}.
\]
Combining these two inequalities gives 
\be\label{eq:C-factor-is-large}
\chi_{\overline{C}, 5\rho}(z^{-1}y^{-1}x) \succcurlyeq
\max\left\{\frac{|\overline{A}|}{|\overline{C}|}, \frac{|\overline{B}|}{|\overline{C}|}\right\}\ge \frac{{|\overline{A}|^{1/2} |\overline{B}|^{1/2}}}{|\overline{C}|}.
\ee
By \eqref{eq:A-factor-is-large}, \eqref{eq:B-factor-is-large}, and \eqref{eq:C-factor-is-large}, Lemma~\ref{lem:MarkovInequality}, Corollary~\ref{cor:larger-nbhd},  for $x\in 1_{\rho}$, we get that 
\[
	\chi_{\overline{A}, 5\rho}\ast \chi_{\overline{B}, 5\rho} \ast \chi_{\overline{C}, 5\rho}(x) 
 	\succcurlyeq |(A_\high)_{ \rho}|  \ |(B_\high)_{ \rho}|  \cdot \frac{{|\overline{A}|^{1/2} |\overline{B}|^{1/2}}}{|\overline{C}|}
	\succcurlyeq
 	\frac{(|\overline{A}||\overline{B}|)^{3/2}}{|\overline{C}|}.
\]	
In order to show $x\in \overline{A}\cdot \overline{B}\cdot \overline{C}$, by \eqref{eq:mixing-for-product}, it suffices to prove that for $\delta$ small enough we have 
\[
\frac{(|\overline{A}||\overline{B}|)^{3/2}}{|\overline{C}|} > \alpha \rho^{\beta} (|\overline{A}||\overline{B}||\overline{C}|)^{-1/2}
\]
where $\alpha$ and $\beta$ are fixed positive numbers that depend on $L$, $d_0$, and other parameters of the group $G$. This inequality holds if and only if  $|\overline{A}||\overline{B}|>\sqrt{\alpha} \rho^{\beta/2} |\overline{C}|^{1/4}$, which, in view of $|\overline{C}|\le 1$, follows from 
\be\label{eq:AB-alpha-beta}
|\overline{A}||\overline{B}|>\sqrt{\alpha} \eta^{(\beta/2)\vare}.
\ee 
Now, recall the condition $\frac{h(A;\eta)+h(B;\eta)}{2}>(1-\delta)h(G;\eta)$. 
This implies 
\be\label{eq:AB-E-delta-d0}
|A_{\eta}||B_{\eta}|\succcurlyeq |1_\eta|^{-2\delta}\succcurlyeq\eta^{2\delta d_0}.
\ee
Consequently, applying Lemma~\ref{lem:basic-properties-metric-entropy} we obtain 
$|\overline{A}||\overline{B}|\ge E^{-1}\eta^{2\delta d_0}$ where $E=\Omega^{O(1)}$.
Finally note that if $\eta^{\vare}\ll_{\alpha,\beta, d_0} 1$, then for small enough $\delta$ we have $E^{-1}\eta^{2\delta d_0}> \sqrt{\alpha} \eta^{(\beta/2)\vare}$. This and~\eqref{eq:AB-E-delta-d0} imply~\eqref{eq:AB-alpha-beta}. The proof is complete.
\end{proof}


\section{A Littlewood-Paley decomposition for locally random groups}\label{s:littlewoord}

In this section, we will give a decomposition of $L^2(G)$ into almost orthogonal subspaces of functions, each consisting of functions {\it living at a different scale}. This notion will be defined later (see Definition 
\ref{def:fun-at-scale}).  We will first treat the case of profinite groups, which is somewhat simpler and sharper results can be obtained. Then, in the next subsection, we will deal with the general case of locally random groups.

\subsection{The case of profinite groups}
Let $G$ be a profinite group, equipped with a bi-invariant metric $d$ such that balls centered at the identity element form a family of normal subgroups. Such a metric always exists. In fact, if $G$ is presented as the inverse limit of finite groups $(G_i)_{i \ge 1 }$,  one can define the distance $d(g,h)$ to be $2^{-i}$ where $i$ is the largest  
index with the property that  $\pi_i(g) = \pi_i(h)$. Here $ \pi_i: G \to G_i$ denotes the natural projection. 

\begin{lem}\label{lem:decomposition-according-to-a-normal-subgroup}
Suppose $G$ is a compact group and $N$ is a normal open subgroup of $G$. Let $f_N:=\frac{\bb1_N}{|N|}$. Then $T_N:L^2(G)\rightarrow L^2(G), T_N(g):=f_N\ast g$ is the orthogonal projection onto the subspace $L^2(G)^N:=\{f\in L^2(G)|\h f(gn)= f(g) \textrm{ for all } n \in N, g \in G\}$ of $N$-invariant functions. In addition 
\[
q:L^2(G)^N\rightarrow L^2(G/N),\quad q(g)(xN):=g(x)
\]
is a well-defined unitary $G$-module isomorphism.
\end{lem}

\begin{proof}
The proof is a standard computation. 
\end{proof}

Given a $G$-valued random variable $X$ with distribution measure $\mu$, let $X_\eta = XZ$, where $Z$ is a random variable with 
distribution $P_{\eta}=\frac{\bb1_{1_{\eta}}}{|1_{\eta}|}$ independent from $X$. 
\begin{lem}\label{eq:density-function-thickened}
Let $\mu_{\eta}$ denote the density function of $X_{\eta}$. Then 
$\mu_{\eta}(x)=\frac{\mu(x_{\eta})}{|1_{\eta}|}$ for all $x\in G$. 
\end{lem}

\begin{proof}
By definition,   for all $f\in C(G)$ we have
\be\label{eq:density-convolution}
\int_G f(x) \mu_{\eta}(x) \d x=\int_G\int_G f(xy) P_{\eta}(y) \d y \d\mu(x).
\ee
Notice that the right hand side of \eqref{eq:density-convolution} is equal to
\begin{align*}
	\int_G\int_G f(z)  P_{\eta}(x^{-1}z) \d z \d\mu(x)=& 
	\int_G f(z) \int_G  P_{\eta}(x^{-1}z)   \d\mu(x) \d z = \int_G f(z) \frac{\mu(z_{\eta})}{|1_{\eta}|} \d z.
\end{align*}
\end{proof}

Define the  \emph{R\'{e}nyi entropy of $X$ at scale} $\eta$ by 
\be\label{eq:renyi-scaled}
H_2(X;\eta):=\log(1/|1_{\eta}|)-\log \|\mu_{\eta}\|_2^2
\ee
 where $\mu$ is the distribution of $X$. We also write $H_2(\mu;\eta)$ instead of $H_2(X;\eta)$. Let us observe that by Lemma \ref{eq:density-function-thickened} we have 
 \[
 \|\mu_{\eta}\|_{\infty}\le 1/|1_{\eta}|; \text{ and so } \|\mu_{\eta}\|_2^2\le  1/|1_{\eta}|,
 \]
  which implies that $H_2(X;\eta)\ge 0$.

\begin{proposition}\label{prop:spec-gap-Littlewood-Paley-profinite}
Suppose $G$ is a compact group with a given bi-invariant metric such that $1_{\eta}$ is a subgroup of $G$ for all $\eta>0$. Suppose $G$ is an $L$-locally random group with coefficient $C_0$. Suppose $G$ satisfies the dimension condition $\Dim(C_1, d_0)$. 
Let $\mu$ be a symmetric Borel probability measure on $G$ whose support generates a dense subgroup of $G$. Fix a number $a>1$ and $\eta_0<1$, and for all $i \ge 1$, let $\eta_i:=\eta_0^{a^i}$  and $\cal_i:=L^2(G)^{1_{\eta_i}}$. Suppose that $C_2>0$ is such that for every $i\gg 1$, there exists an integer $l_i\le C_2 h(G;\eta_i)$ such that 
\[
\text{ (Large entropy at scale } \eta \text{)}\hspace{1cm} H_2(\mu^{(l_i)};\eta_i)\ge  \left( 1-\frac{1}{8Ld_0a}\right) h(G;\eta_i).
\]
Then there exists $i_0 \ge 1$ such that 
\[
\lcal(\mu;L^2(G)\ominus \cal_{i_0}) \ge \frac{1}{16C_2Ld_0a}.
\]
In particular, $\lcal(\mu;G)>0$. 
\end{proposition}
\begin{proof}
For all $i$ and $f\in \cal_{i+1}\ominus \cal_i$, we have $f_{\eta_i}=0$ and $f_{\eta_{i+1}}=f$. Hence for $f\in 	\cal_{i+1}\ominus \cal_i$ and every symmetric Borel probability measure $\nu$ we have $\|\nu\ast f\|_2=\|\nu_{\eta_{i+1}}\ast f_{\eta_{i+1}}\|_2$. Applying Theorem~\ref{thm:mixing-inequality} for $i \gg 1$ we obtain
\begin{align*}
\|\nu\ast f\|_2^2\le & 2\|\nu_{\eta_i}\|_2^2 \|f_{\eta_i}\|_2^2+\eta_i^{1/(2L)} \|\nu_{\eta_{i+1}}\|_2^2 \|f_{\eta_{i+1}}\|_2^2 \\
=& \eta_i^{1/(2L)} \|\nu_{\eta_{i+1}}\|_2^2 \|f\|_2^2=\eta_{i+1}^{1/(2La)} \|\nu_{\eta_{i+1}}\|_2^2 \|f\|_2^2.	
\end{align*}
This implies
\[
2\lcal(\nu; \cal_{i+1}\ominus \cal_i)\ge H_2(\nu;\eta_{i+1})-h(G;\eta_{i+1})-\frac{1}{2La} \log \eta_{i+1}.
\]
Since $1_{\eta}$ is a group, $h(G;\eta)=\log(1/|1_{\eta}|)$; and so   
by the dimension condition we have 
\[
|h(G;\eta)+d_0 \log \eta|\le \log C_1.
\]
Therefore by the previous inequality, for $\eta_{i+1}\ll_{C_1} 1$, we have
\[
2\lcal(\nu; \cal_{i+1}\ominus \cal_i)\ge H_2(\nu;\eta_{i+1})- \left( 1-\frac{1}{4Ld_0a} \right)  h(G;\eta_{i+1}).
\]
Applying the above inequality for $\nu:=\mu^{(l_i)}$ coupled with $\lcal(\mu^{(l_i)};\cal_{i+1}\ominus \cal_i)=l_i \lcal(\mu;\cal_{i+1}\ominus \cal_i)$ implies that for $i\gg 1$ we have
\[
\lcal(\mu;\cal_{i+1}\ominus \cal_i)\ge \frac{1}{16C_2Ld_0a}.
\]
As a result, $\lcal(\mu; L^2(G)\ominus \cal_{i_0})\ge \frac{1}{16C_2Ld_0a}$ for some $i_0$. Since $\cal_{i_0}\ominus \bbc \bb1_G$ is a finite dimensional subspace of $L^2_0(G)$, and the support of 
$\mu$ generates a dense subgroup, we have $\lcal(\mu;G)>0$.
\end{proof}

Now we interpret the spaces $\cal_{i+1}\ominus \cal_i$'s in terms of certain convolution operators. This point of view will be extended to an arbitrary locally random group.
\begin{lem}\label{lem:decomposition-according-to-a-filteration}
Suppose $G$ is a compact group, $G:= N_1\supseteq N_2\supseteq \cdots $ is a sequence of normal open subgroups of $G$ that form a basis for the neighborhoods of $1$. For integers $i \ge 1$, let 
\[
\Delta_i:L^2(G)\rightarrow L^2(G), \quad  \Delta_i(g):=f_{N_{i+1}}\ast g-f_{N_{i}}\ast g,
\]
and let $\Delta_0(g):=f_{N_1}\ast g$. Then the following statements hold.
\begin{enumerate}
\item For all $g\in L^2(G)$ we have $g=\sum_{i=0}^{\infty} \Delta_i(g)$ in $L^2(G)$.	
\item For all $i\neq j$ and $g\in L^2(G)$, we have $\Delta_i(g)\perp \Delta_j(g)$.
\item For all $g\in L^2(G)$ we have $\|g\|_2^2=\sum_{i=0}^{ \infty} \|\Delta_i(g)\|_2^2$
\item If $\mu$ is a Borel probability measure on $G$, then $\Delta_i(\mu\ast f)=\mu\ast \Delta_i(f)$ for all $i$.
\end{enumerate}
\end{lem}
\begin{proof}
For every integer $k \ge 1$, define $\cal_k:=L^2(G)^{N_k}$. Since $G:=N_1\supseteq N_2\supseteq \cdots$, we have
$
\bbc \bb1_G=\cal_1\subseteq \cal_2 \subseteq \cdots.
$ 
By Lemma~\ref{lem:decomposition-according-to-a-normal-subgroup}, we have that $f_{N_{j}}\ast g$ is the orthogonal projection of $g$ onto $\cal_{j}$ for any $j$. And so $\Delta_i(g)\in \cal_{i+1} \ominus \cal_{i}$ for positive integer $i$, and $\Delta_0(g)\in \cal_1$. This implies (2).   

Every element $g$ of the matrix algebra $\ecal(G)$ generates a finite dimensional $G$-submodule $M$ of $L^2(G)$, defining a unitary representation $\pi_M:G\rightarrow \ucal(M)$. Since $G$ is profinite, we have that $\pi_M(G)$ is a finite group. Hence $\ker \pi_M$ is an open subgroup of $G$. Therefore, $N_k\subseteq \ker \pi_M$ for some $k$,  implying $g\in \cal_k$. It follows that  
$g=\sum_{i=0}^{j}\Delta_i(g)$ for all $j\ge k-1$. 
By the Peter-Weyl Theorem $\ecal(G)$ is dense in $L^2(G)$, from which part (1) follows. Part (3) is an immediate implication of (1) and (2).  
%
%

In order to prove (4), note that if $f$ is a class function, then
\[	\mu\ast (f \ast g)	=  f\ast (\mu \ast g). \]
Since $f_{N_j}$'s are class functions, the claim follows. 
\end{proof}
\begin{remark} 
It follows from the above argument that $\ecal(G)=\bigcup_{i=1}^{\infty} L^2(G)^{N_i}$.
\end{remark}

\subsection{The general case}
 
In the rest of this section we will prove a generalization of Lemma~\ref{lem:decomposition-according-to-a-filteration} that applies to general locally random groups. The results of this
section will be crucially used in the next section to prove a generalization of Proposition~\ref{prop:spec-gap-Littlewood-Paley-profinite}. Another result of this section, Proposition~\ref{prop:essential-support-fourier-fun-scale}, is 
a Fourier theoretic interpretation of the notion of {\it living at a given scale} (see Definition 
\ref{def:fun-at-scale} for definition), which parallels the classical Paley-Littlewood theory. 

A major difficulty in dealing with the general case is that unlike profinite groups, neighborhoods of identity are only {\em approximate} subgroups in general compact groups. Throughout this section, we will assume that the group $G$ satisfies the following two properties:

\begin{enumerate}
\item $G$ is a compact group which is $L$-locally random with coefficient $C_0$. 
\item \ref{eq:dimension-condition-intro}$(C_1, d_0)$: for all $\eta >0$
$$C_1^{-1} \eta^{d_0} \le |1_{\eta}|\le C_1 \eta^{d_0}.$$
\end{enumerate}

As in Proposition~\ref{prop:spec-gap-Littlewood-Paley-profinite} we let $\eta_0$ be a small positive number, whose value will be specified later. Also fix 
$$a\ge 4Ld_0, \qquad  \text{and set } \quad \eta_i:=\eta_0^{a^i}, \text{for } \, i \ge 1. $$ 
As in Lemma~\ref{lem:decomposition-according-to-a-filteration}, we define a family of operators 
$\Delta_i:L^2(G)\rightarrow L^2(G)$ by  setting 
$\Delta_0(g):=P_{\eta_0}\ast g$ and for every $i \ge 1$
 \begin{equation}\label{Deltai}
\Delta_i(g):=(P_{\eta_{i+1}}-P_{\eta_i})\ast g.
\end{equation} 
Since $P_{\eta}$ is invariant under conjugation, $\Delta_i$'s commute with any convolution operator (including convolution by a Borel probability measure), and for all $x,x'\in G$ we have 
\[
\lambda(x)\circ \rho(x') \circ \Delta_i =\Delta_i\circ \lambda(x)\circ \rho(x'), 
\]
where $\lambda$ and $\rho$ denote, respectively, the left and right-regular representations of $G$. 

We showed previously that if $1_{\eta}$'s are subgroups, then $ 
(\Delta_i(g))_{\eta_i}=0$ and $(\Delta_i(g))_{\eta_{i+1}}=\Delta_i(g)$.
We start by showing an {approximate} version of these equalities. 
 In this section, we only establish properties of the operators $\Delta_i$'s, and  postpone the discussion on their connections with spectral gap properties of $T_{\mu}$ to the next section. 
    
\begin{proposition}\label{prop:image-of-operators-functions-at-scale-eta}
In the setting of this section, if integers $i, j, k$ satisfy $ 0 \le j < i $ and $k>i+1$, then the following hold:
\begin{itemize}
\item (Averaging to zero) $\|\Delta_i(g)_{\eta_j}\|_2\ll_{C_0,C_1,L} \eta_0^{a^{i}/(4L+2)} \|g\|_2$. 
\item (Almost invariant) $\|\Delta_i(g)_{\eta_k}-\Delta_i(g)\|_2\le 2 \eta_0^{a^k/(8L)} \|g\|_2$.	
\end{itemize}
\end{proposition}
\begin{proof}
The argument for the first part is fairly similar to the one presented for Lemma~\ref{lem:almost-orthogonality}. We let $D$ be a threshold parameter whose value will be set later and estimate the corresponding low frequency and high frequency terms. 
By Lemma~\ref{lem:low-freq} we have 
\[
\|\wh{P}_{\eta}(\pi)-I\|_{\op} \le C_0 (\dim \pi)^L \eta.
\]
Combined with the trivial bound $\|\wh{P}_{\eta}(\pi)\|_{\op}\le 1$, this implies
\begin{align}
\notag	L(\Delta_i(g)_{\eta_j};D)= & \sum_{\pi\in \wh{G}, \dim \pi\le D} \dim \pi \h \Bigl\|\wh{P}_{\eta_j}(\pi) \bigl(\wh{P}_{\eta_{i+1}}(\pi)-\wh{P}_{\eta_i}(\pi)\bigr) \wh{g}(\pi)\Bigr\|_{\HS}^2 \\
\label{eq:ave-zero-image-operator-low-freq}	\le & \h C_0^2 D^{2L} (\eta_{i+1}+\eta_i)^2 L(g;D) \le 4C_0^2 D^{2L} \eta_i^2 \|g\|_2^2.
\end{align}
For the high frequency term, by Lemma~\ref{lem:low-and-high-covolution} and the trivial bound 
$\|\wh{P}_{\eta_{i+1}}(\pi)-\wh{P}_{\eta_{i}}(\pi)\|_{\op}\le 2$, we have
\begin{align}
\notag	H(\Delta_i(g)_{\eta_j};D)= & \sum_{\pi\in \wh{G}, \dim \pi> D} \dim \pi \h \Bigl\|\wh{P}_{\eta_j}(\pi) \bigl(\wh{P}_{\eta_{i+1}}(\pi)-\wh{P}_{\eta_i}(\pi)\bigr) \wh{g}(\pi)\Bigr\|_{\HS}^2 \\
\label{eq:ave-zero-image-operator-high-freq}	\le & \h \frac{4}{D} H(P_{\eta_j};D)H(g;D) \le\frac{4}{D|1_{\eta_j}|} \|g\|_2^2 
\le \frac{4C_1}{D\eta_j^{d_0}}\|g\|_2^2.
\end{align}
We choose $D$ such that $4C_0^2 D^{2L} \eta_i^2=\frac{4C_1}{D\eta_j^{d_0}}$, which implies that $D$ equals $\eta_j^{-d_0/(2L+1)}\eta_i^{-2/(2L+1)}$ up to a multiplicative factor, which is a function of the constants $C_0,C_1,$ and $L$. Hence by \eqref{eq:ave-zero-image-operator-low-freq} and \eqref{eq:ave-zero-image-operator-high-freq} we get
\[
\|\Delta_i(g)_{\eta_j}\|_2^2\ll_{C_0,C_1,L} \eta_j^{-d_0+d_0/(2L+1)}\eta_i^{2/(2L+1)} \|g\|_2^2.
\]
Notice that 
\[
\eta_j^{-d_0+d_0/(2L+1)}\eta_i^{2/(2L+1)}=\eta_0^{\frac{2}{2L+1} a^i- \frac{2Ld_0}{2L+1} a^j};
\]
and $a^i-(2Ld_0)a^j\ge a^i(1-(2Ld_0)a^{-1})\ge a^i/2$. Therefore, 
\[
\|\Delta_i(g)_{\eta_j}\|_2^2\ll_{C_0,C_1,L} \eta_0^{a^i/(2L+1)} \|g\|_2^2;
\]
and the first part follows.

For the second part we use Lemma~\ref{lem:almost-orthogonality} to obtain
\begin{align*}
	\|\Delta_i(g)_{\eta_k}-\Delta_i(g)\|_2 = &
	\|\Delta_i(g_{\eta_k}-g)\|_2 \\
	\le & \|(g_{\eta_k}-g)_{\eta_{i+1}}\|_2+  \|(g_{\eta_k}-g)_{\eta_{i}}\|_2 \\
	\le & 2 \eta_k^{1/(8L)} \|g\|_2.
\end{align*}
\end{proof}

\begin{definition}\label{def:fun-at-scale}
We say $g\in L^2(G)$ \emph{lives at scale} $\eta$ \emph{(with parameter $a$)} if 
\begin{itemize}
	\item (Averaging to zero) $\|g_{\eta^{1/a}}\|_2\le \eta^{1/(2a)} \|g\|_2$.
	\item (Almost invariant) $\|g_{\eta^{a^2}}-g\|_2\le \eta^{a/2} \|g\|_2$.
\end{itemize} 
\end{definition}

From Proposition~\ref{prop:image-of-operators-functions-at-scale-eta} we deduce that if $\|\Delta_i(g)\|_2/\|g\|_2\gg 1$, then $\Delta_i(g)$ lives at scale $\eta_i$. The next proposition 
provides a Fourier theoretic understanding of this notion. 

For every $\pi\in \wh{G}$, let $H_{\pi}$ denote the subspace of $L^2(G)$ spanned by the matrix coefficients of $\pi$.
Given an interval $I\subset \bbr$, set 
\[
\cal_I:=\oplus_{\pi\in \wh{G}, \dim \pi\in I} H_{\pi},
\]
and denote by $\pi_I:L^2(G)\to \cal_I$ the
corresponding orthogonal projection. 

\begin{proposition}\label{prop:essential-support-fourier-fun-scale}
Let $0<\eta<1$ be a parameter. 
\begin{enumerate}
\item Suppose $f\in L^2(G)$ lives at scale $\eta$. Then 
\[
\|\pi_{I_\eta}(f)\|_2^2 \geq \bigl( 1-8\eta^{1/(2a)} \bigr)  \|  f \|^2_2.
\] 
where $I_\eta=[\frac{1}{2C_0}\eta^{-1/(La)},2C_0 \eta^{-d_0a^2}]$.  
\item Let $I'_\eta=[C_1\eta^{-\frac{d_0+1}{a}}, C_0^{\frac{-1}{L}}\eta^{\frac{-2a^2+a}{2L}}]$. Then every $f\in \cal_{I'_\eta}$ 
lives at scale $\eta$. 
\end{enumerate}	
\end{proposition}

\begin{proof}
Without loss of generality, assume that $ \| f \|_2=1$. To see part~(1) it suffices to show that  
\begin{equation}\label{eq:bound-Fourier-scale}
L(f; {(2C_0)^{-1}}\eta^{-1/(La)})\le 4 \eta^{1/2a}
\text{ and } \quad
H(f; 2C_0 \eta^{-d_0a^2})\le 4 \eta^{a/2}. 
\end{equation}

By Lemma~\ref{lem:norm-eta-averaging}, for an arbitrary threshold $D$ satisfying $ C_0D^L\eta^{1/a} <1$, we have 
\[
L(f;D)\le (1-C_0D^L\eta^{1/a})^{-2} L(f_{\eta^{1/a}};D)\le (1-C_0D^L\eta^{1/a})^{-2} \eta^{1/(2a)}.
\]
In the last inequality we used $\|f_{\eta^{1/a}}\|_2\leq \eta^{1/(2a)}\|f\|_2$, which holds since $f$ lives at scale $\eta$.
Setting $D:=\frac{1}{2C_0}\eta^{-1/(La)}$, the first inequality in~\eqref{eq:bound-Fourier-scale} follows.

To show the second inequality  in~\eqref{eq:bound-Fourier-scale}, we note that  
\be\label{eq:f2-feta-2}
\|f\|_2^2-\|f_{\eta^{a^2}}\|_2^2=(\|f\|_2-\|f_{\eta^{a^2}}\|_2)(\|f\|_2+\|f_{\eta^{a^2}}\|_2)\le 2\|f-f_{\eta^{a^2}}\|_2\le 2\eta^{a/2}.
\ee
Since $\|P_{\eta^{a^2}}\|_1=1$, for all $\pi\in \wh{G}$ we have $\|\wh{P}_{\eta^{a^2}}(\pi)\|_{\op}\le 1$. 
In consequence, Lemma~\ref{lem:low-and-high-covolution} implies that for an arbitrary threshold $D'$ we have
\[
L(f;D')-L(f_{\eta^{a^2}};D')\ge 0.
\]
This and~\eqref{eq:f2-feta-2} imply that 
\[
H(f;D')-H(f_{\eta^{a^2}};D')\le 2\eta^{a/2}.
\]
Altogether, we deduce 
\begin{align*}
	H(f;D')\le &2\eta^{a/2}+H(f_{\eta^{a^2}};D')\\ 
	\le  & 2\eta^{a/2}+\frac{1}{D'} H(P_{\eta^{a^2}};D')H(f;D') && (\text{by Lemma~\ref{lem:low-and-high-covolution}}) \\
	\le & 2\eta^{a/2}+ \frac{1}{D'|1_{\eta^{a^2}}|} H(f;D')&& (\text{by } H(P_{\eta^{a^2}};D')\leq \|1_{\eta^{a^2}}\|_2^2) \\
	\le & 2\eta^{a/2}+\frac{C_0}{D'\eta^{d_0a^2}} H(f;D').
\end{align*}
Therefore $\Big(1-\frac{C_0}{D'\eta^{d_0a^2}} \Bigr) H(f;D')\le 2\eta^{a/2}$. 
Setting $D':= 2C_0\eta^{-d_0a^2}$, the claim in part~(1) follows.

\medskip

We now turn to part~(2). Let $f\in\cal_{I'_\eta}$ be a unit vector. Note that for every $\pi$ with $\dim\pi\not\in I'_\eta$, $\hat f(\pi)=0$. 
In particular, $L(f;D)=0$ for any $D<C_1\eta^{-\frac{d_0+1}{a}}$. 
Therefore, by Lemma~\ref{lem:low-and-high-covolution}, we have  
\begin{align*}
\|f_{\eta^{1/a}}\|_2^2=\|P_{\eta^{1/a}}\ast f\|_2^2&\leq C_1^{-1}{\eta^{(d_0+1)/a}}\|P_{\eta^{1/a}}\|_2^2\|f\|_2^2\\
&\leq C_1^{-1}{\eta^{(d_0+1)/a}}\frac{1}{|1_{\eta^{1/a}}|}\leq \eta^{1/a};
\end{align*}
we used~\eqref{eq:dimension-condition-intro} in the second inequality.

To verify the required bound for $\|f_{\eta^{a^2}}-f\|_2$, we use Lemma~\ref{lem:low-freq} combined with 
the fact that for every $\pi$ with $\dim\pi\not\in I'_\eta$, $\hat f(\pi)=0$, and conclude that 
\begin{align*}
\|f_{\eta^{a^2}}-f\|_2^2&= \sum_{\dim\pi\in I'_\eta}\dim(\pi)\|(I-\hat P_{\eta^{a^2}}(\pi))\hat f(\pi)\|_{{\rm HS}}^2\\
&\leq\sum_{\dim\pi\in I'_\eta}\dim(\pi)\|I-\hat P_{\eta^{a^2}}(\pi)\|_{\op}^2\|\hat f(\pi)\|_{{\rm HS}}^2\\
&\leq \sum_{\dim\pi\in I'_\eta}C_0^2\dim(\pi)^{2L}\eta^{2a^2}\dim(\pi)\|\hat f(\pi)\|_{{\rm HS}}^2\leq \eta^{a}.
\end{align*}

This completes the proof of part~(2) and the lemma.
\end{proof}

We will now prove an almost orthogonality of the images of $\Delta_i$'s and show that their sum is dense in $L^2(G)$.

\begin{lem}\label{lem:almost-orthogonality-of-image-operators}
In the setting of this section, for non-negative integers $j<i-1$, and $g\in L^2(G)$ we have 
\[
\|\Delta_i \Delta_j\|_{\op}\ll_{C_0,C_1,L} \eta_i^{1/(4L+2)}\text{ and  } 
|\langle \Delta_i(g),\Delta_j(g)\rangle|\ll_{C_0,C_1,L} \eta_i^{1/(4L+2)} \|g\|_2^2.
\]
\end{lem}
\begin{proof}
Since $\Delta_i$ is a self-adjoint operator, we have $\langle \Delta_i(g),\Delta_j(g)\rangle=\langle g, \Delta_i(\Delta_j(g))\rangle$; this implies 
\[
|\langle \Delta_i(g),\Delta_j(g)\rangle|\le \|\Delta_i \Delta_j\|_{\op}\|g\|_2^2.
\]
By the first part of Proposition~\ref{prop:image-of-operators-functions-at-scale-eta} for $j>0$ we have
\[
\|\Delta_i \Delta_j(g)\|_2 =  \|\Delta_i(g)_{\eta_{j+1}}-\Delta_i(g)_{\eta_{j}}\|_2 \\
\ll_{C_0,C_1,L} \eta_i^{1/(4L+2)} \|g\|_2.
\]
For $j=0$ it is similar and  the claims follow.
\end{proof}
\begin{lem}\label{lem:decompositing-g-using-operators}
In the setting of this section, $g=\sum_{i=0}^{\infty} \Delta_i(g)$ for any $g\in L^2(G)$.	
\end{lem}
\begin{proof}
It suffices to show that for all $g\in L^2(G)$, $ \| g-\sum_{i=1}^n \Delta_i(g)  \|_2=  \| g-g_{\eta_{n+1}} \|_2$ tends to zero as $n \to \infty$.  By the Peter-Weyl Theorem, for every $\vare>0$ there is $f\in C(G)$ such that $\|f-g\|_2\le \vare$. Since $G$ is compact, $f$ is uniformly continuous. Let $\eta>0$ be such that
\[
d(x,y)\le \eta \text{ implies that } |f(x)-f(y)|\le \vare.
\]
For $n \gg_{ \vare} 1$, we have  $\|f_{\eta_n}-f\|_{\infty}\le \vare$. Hence $\|f_{\eta_n}-f\|_2\le \vare$. On the other hand, $\|f-g\|_2\le \vare$ implies that $\|f_{\eta_n}-g_{\eta_n}\|_2\le \vare$. Therefore for $n  \gg_{ \vare} 1$ we have 
\[
\|g-g_{\eta_n}\|_2\le \|g-f\|_2+\|f-f_{\eta_n}\|_2+\|f_{\eta_n}-g_{\eta_n}\|_2\le 3\vare.
\]
Thus $\lim_{n \to \infty} g_{\eta_n}=g$ in $L^2$, from which the claim follows. 
\end{proof}
By a similar argument as in the proof of the Cotlar-Stein Lemma (see \cite[Lemma~6.3]{BIG-17}, and also \cite[Chapter VII]{Stein93}), we will prove 
\begin{proposition}\label{prop:sum-norm-squared-decomposition-l2-norm}
In the setting of this section, for $\eta_0\ll_{C_0,C_1,L} 1$, and $g\in L^2(G)$ we have 
\begin{equation}\label{eq:sum-sqaures-and-norm-littlewood-payley}
\|g\|_2^2 \ll \sum_{i=0}^{\infty} \|\Delta_i(g)\|_2^2 \ll \|g\|_2^2.
\end{equation}
\end{proposition}

In preparation for the proof we will need to establish some inequalities. 
\begin{lem}\label{lem:upper-bound-needed-cotlar-stein}
In the setting of this section, for a non-negative integer $i$, we have
\[
\sum_{j=0}^{\infty} \|\Delta_i \Delta_j\|_{\op}^{1/2} \ll_{C_0,C_1,L} 1.
\]	
\end{lem}
\begin{proof}
By Lemma~\ref{lem:almost-orthogonality-of-image-operators} and $\|\Delta_j\|_{op}\le 2$, we get that
\[
\sum_{j=0}^{\infty} \|\Delta_i \Delta_j\|_{\op}^{1/2} \le 6+ O_{C_0,C_1,L}\Bigl(\sum_{j=1}^{\infty} \eta_0^{a^j/(4L+2)}\Bigr)\ll_{C_0,C_1,L} 1.
\]	
\end{proof}
The proof of the next lemma is based on the proof of the Cotlar-Stein lemma.
\begin{lem}\label{lem:cotlar-stein}
 In the above setting, for every $g\in L^2(G)$, we have 
 \[
 \sum_{i,j} |\langle \Delta_i(g),\Delta_j(g)\rangle| \ll \|g\|_2^2.
 \]	
\end{lem}
\begin{proof}
For a given $g\in L^2(G)$, for every $i\neq j$, choose $u_{i,j}\in \mathbb{S}^1\cup\{0\}$ such that $|\langle \Delta_i(g),\Delta_j(g)\rangle|=u_{i,j} \langle \Delta_i(g),\Delta_j(g)\rangle$ where $u_{i,j}=0$ if $\langle \Delta_i(g),\Delta_j(g)\rangle=0$. Then for every integer $N \ge 1$ we have 
\[
\sum_{0\le i,j\le N} |\langle \Delta_i(g),\Delta_j(g)\rangle|=\langle R_N(g),g\rangle,
\]
where $R_N=\sum_{0\le i,j\le N} u_{i,j} \Delta_j \Delta_i$. Thus, it is enough to prove that for all possible choices of $u_{i,j}$ and all $N \ge 1$ we have $\|R_N\|_{\op}\le \Phi$ for a fixed positive number $\Phi$. Since $\Delta_i$'s are self-adjoint and pairwise commuting, for every positive integer $k$ we have $\|R_N^k\|_{\op}=\|R_N\|_{\op}^k$. By the triangle inequality, we have
\[
	\|R_N\|_{op}^k\le  \sum_{0\le i_l,j_l\le N, \forall 1\le l\le k} \|\Delta_{i_1}\Delta_{j_1}\cdots\Delta_{i_k}\Delta_{j_k}\|_{\op}. 
\]
Since \[ \|\Delta_{i_1}\Delta_{j_1}\cdots\Delta_{i_k}\Delta_{j_k}\| \le \min  \biggl( \prod_{l=1}^{k}\|\Delta_{i_l}\Delta_{j_l}\|_{\op}, 
\|\Delta_{i_1}\|_{\op}\|\Delta_{j_k}\|_{\op} \prod_{l=1}^{k-1} \|\Delta_{j_l}\Delta_{i_{l+1}}\|_{\op} \biggr), \]
we have that 
\[
\|\Delta_{i_1}\Delta_{j_1}\cdots\Delta_{i_k}\Delta_{j_k}\|_{\op} \le 4 \biggl(\prod_{l=1}^k \|\Delta_{i_l}\Delta_{j_l}\|_{\op} \prod_{l=1}^{k-1} \|\Delta_{j_l} \Delta_{i_{l+1}}\|_{\op}\biggr)^{1/2}.
\]
Altogether we get 
\begin{equation}
\begin{split}      
\| R_N\|_{\op}^k & \le 4 \sum_{ i_1=0}^{ N}   \sum_{ j_1=0}^{ N}  \cdots  \sum_{ j_k=0}^{ N} 
 \biggl(\prod_{l=1}^k \|\Delta_{i_l}\Delta_{j_l}\|_{\op} \prod_{l=1}^{k-1} \|\Delta_{j_l} \Delta_{i_{l+1}}\|_{\op}\biggr)^{1/2} \\
& = 4 \sum_{ i_1=0}^{ N}   \sum_{ j_1=0}^{ N}  \cdots  \sum_{ i_k=0}^{ N}   
 \biggl(\prod_{l=1}^{k-1} \|\Delta_{i_l}\Delta_{j_l}\|_{\op} \prod_{l=1}^{k-1} \|\Delta_{j_l} \Delta_{i_{l+1}}\|_{\op}\biggr)^{1/2}
\biggl( \sum_{ j_k=0}^{ N}  \| \Delta_{i_k} \Delta_{j_k} \|_\op^{1/2}\biggr).  \\\end{split}
\end{equation}
By repeatedly using Lemma~\ref{lem:upper-bound-needed-cotlar-stein}, it follows that there is a constant $M:=M(C_0,C_1,L)$ such that 
\[
\|R_N\|_{\op}^k\le 4 (N+1) M^{2k-1},
\] 
which implies $\|R_N\|_{\op}\le 4^{1/k} (N+1)^{1/k} M^2$ for any positive integer $k$. The claim follows from here. 
\end{proof}
\begin{corollary}\label{cor:norm-operators-l2-series}
In the setting of this section, for $g\in L^2(G)$ we have that 
\[
\sum_{i=0}^{\infty} \|\Delta_i(g)\|_2^2\ll \|g\|_2^2, \emph{ and }	 
\sum_{i=0}^{\infty} |\langle \Delta_i(g),\Delta_{i+1}(g)\rangle| \le \sum_{i=0}^{\infty} \|\Delta_i(g)\|_2^2.
\]
\end{corollary}
\begin{proof}
The first inequality is a weaker version of the inequality given in Lemma~\ref{lem:cotlar-stein}. Applying the Cauchy-Schwarz inequality twice, we obtain
\begin{align*}
\sum_{i=0}^{\infty} |\langle \Delta_i(g),\Delta_{i+1}(g)\rangle|
\le & \sum_{i=0}^{\infty} \|\Delta_i(g)\|_2\|\Delta_{i+1}(g)\|_2 \\
\le & \biggl( \sum_{i=0}^{\infty} \|\Delta_i(g)\|_2^2\biggr)^{1/2} \biggl( \sum_{i=0}^{\infty} \|\Delta_{i+1}(g)\|_2^2\biggr) ^{1/2} \\
\le & \sum_{i=0}^{\infty} \|\Delta_i(g)\|_2^2.
\end{align*}
\end{proof}

\begin{proof}[Proof of Proposition \ref{prop:sum-norm-squared-decomposition-l2-norm}]
By Lemma~\ref{lem:decompositing-g-using-operators} we have $g=\sum_{i=1}^{\infty} \Delta_i(g)$. It follows that 
\begin{align}
\notag	\|g\|_2^2=& \sum_{0\le i,j} \langle \Delta_i(g),\Delta_j(g)\rangle \\
\notag	= & \sum_{i=0}^{\infty} \|\Delta_i(g)\|_2^2 + 2 \sum_{i=0}^{\infty} \langle \Delta_i(g),\Delta_{i+1}(g) \rangle+
	 2 \sum_{0\le i<j, |i-j|>1}\langle \Delta_i(g),\Delta_{j}(g) \rangle
	\\
\label{eq:one-shift-upper-bound} \le & 3 \sum_{i=0}^{\infty} \|\Delta_i(g)\|_2^2 +  2 \sum_{0\le i<j, |i-j|>1}|\langle \Delta_i(g),\Delta_{j}(g) \rangle|
\\
\label{eq:almost-orthogonality-larger-than-one-shift} \le & 3 \sum_{i=0}^{\infty} \|\Delta_i(g)\|_2^2 + O_{C_0,C_1,L}\Bigl(\sum_{0\le i<j, |i-j|>1} \eta_0^{a^j/(4L+2)}) \|g\|_2^2\Bigr)
\\
\notag \le & 3 \sum_{i=0}^{\infty} \|\Delta_i(g)\|_2^2 + O_{C_0,C_1,L}(\eta_0^{a/(4L+2)}) \|g\|_2^2 \\
\notag \le & 3 \sum_{i=0}^{\infty} \|\Delta_i(g)\|_2^2 + (1/2) \|g\|_2^2,
\end{align}
where \eqref{eq:one-shift-upper-bound} is deduced from Corollary~\ref{cor:norm-operators-l2-series} and \eqref{eq:almost-orthogonality-larger-than-one-shift} follows from Lemma~\ref{lem:almost-orthogonality-of-image-operators}. The reverse inequality is already proven in Corollary \ref{cor:norm-operators-l2-series}.
\end{proof}

\section{Littlewood-Paley decomposition and spectral gap}\label{s:spec}
The main goal of this section is to prove Theorem~\ref{thm:large-renyi-entropy} which is a generalization of Proposition~\ref{prop:spec-gap-Littlewood-Paley-profinite} for general locally random groups. At the end, we will show how the existence of spectral gap can be reduced to study of the gap for functions that live at small scales, Theorem~\ref{thm:functions-scale-spec-gap}.


We continue to assume that $G$ is a compact group satisfying the following two properties:
\begin{enumerate}
\item $G$ is  an $L$-locally random group with coefficient $C_0$. 
\item \ref{eq:dimension-condition-intro}$(C_1, d_0)$: for all $\eta >0$
$$C_1^{-1} \eta^{d_0} \le |1_{\eta}|\le C_1 \eta^{d_0}.$$
\end{enumerate}
Fix $a>\max(4Ld_0,4L+2)$, and set $\eta_0$ to be a sufficiently small positive number whose value will be determined later and $\eta_i:=\eta_0^{a^i}$. Define $(\Delta_j)_{j \ge 0}$ as in \eqref{Deltai}.  We begin with a basic property of these operators. 

\begin{lem}\label{lem:orthonormal-eigenbasis}
For all $j \ge 0$, $\Delta_j$ is a compact operator. Moreover, for any symmetric Borel probability measure $\mu$ on $G$, 
there exists an orthonormal basis $\{e_i\}_{i=1}^{\infty}$ of $L^2(G)$ consisting of common eigenfunctions of $\{\Delta_j: j\geq0\}$ and $T_{\mu}$. 
\end{lem}

\begin{proof}
Since $\Delta_j$ is a convolution operator by a function in $L^2(G)$, it is a compact operator. Further, since $1_{\eta}$ is a symmetric subset, $\Delta_j$ is a self-adjoint operator. 


The construction of an orthonormal basis consisting of eigenvectors for $\{\Delta_j\}$ and $T_\mu$ follows from standard arguments
in view of commutativity of the family, compactness of $\{\Delta_j\}$, and the fact that $f=\sum_{j=0}^{\infty} \Delta_j(f)$ for any $f\in L^2(G)$. 
%
%
%
\end{proof}

\begin{lem}\label{lem:eigenfunctions-live-at-certain-scale}
In the setting of this section, suppose $\{e_i\}	_{i=1}^{\infty}$ is an orthonormal basis of $L^2(G)$ which consists of common eigenfunctions of $\Delta_j$'s (see Lemma~\ref{lem:orthonormal-eigenbasis}). Suppose $\Delta_j(e_i)=\alpha_{ji} e_i$ for all $i \ge 1$ and $j \ge 0$. Then 
\begin{itemize}
\item $\|(e_i)_{\eta_{j-1}} \|_2 \ll_{C_0,C_1,L} |\alpha_{ji}|^{-1} \eta_j^{1/(4L+2)}$.
\item $\|(e_i)_{\eta_{j+2}}-e_i\|_2 \le 2 |\alpha_{ji}|^{-1} \eta_{j+2}^{1/(8L)}$.
\end{itemize}
In particular, if $|\alpha_{ji}|\ge \eta_j^{1/(8L+4)}$, then $e_i$ lives at scale $\eta_j$.  
\end{lem}

\begin{proof}
This is an immediate consequence of Proposition~\ref{prop:image-of-operators-functions-at-scale-eta}.	
\end{proof}

\begin{proof}[Proof of Theorem~\ref{thm:large-renyi-entropy}]
We will use the above notation. 
	Let $I_j:=\{i\in \bbz^+|\h |\alpha_{ji}|\ge \eta_j^{1/(8L+4)}\}$, $E:=\bbz^+\setminus \bigcup_{j=1}^{\infty} I_j$, and for $i\in I_j$ we let $\cal_{ji}:=\ker(\Delta_j-\alpha_{ji}I)$. 
	
	We will show the claim holds with $\cal_0$ the space spanned by $\{e_i:i\in E\}$. Let us first show that 
	$\cal_0$ is finite dimensional. By definition, for all $i\in E$ and all positive integers $j$, we have 
\[
|\alpha_{ji}|\le \eta_j^{1/(8L+4)}.
\]
On the other hand, by Lemma~\ref{lem:decompositing-g-using-operators} we have 
$
\sum_{j=0}^{\infty} \alpha_{ji}=1.
$
Therefore 
\[|1-\alpha_{0i}|\le \sum_{j=1}^{\infty} \eta_j^{1/(8L+4)}\le \eta_0^{1/(8L+4)}.\]
Therefore $\alpha_{0i}>1-\eta_0^{1/(8L+4)}$ for any $i\in E$. 

Notice that $\Delta_0$ is a Hilbert-Schmidt operator with kernel $k(x,y):=P_{\eta_0}(xy^{-1})$. Therefore $P_{\eta_0}(xy^{-1})=\sum_{i} \alpha_{0i} e_i(x)\overline{e_i(y)}$. This implies that 
\[
\frac{1}{|1_{\eta_0}|}=\int_G \int_G P_{\eta_0}(xy^{-1})^2 \d y \d x =\sum_i |\alpha_{0i}|^2.
\]
By the above equality, we get 
\[
(1-\eta_0^{1/(8L+4)})^2 \, \#E \le \frac{1}{|1_{\eta_0}|};
\]
which implies that $\dim \cal_0\le  \frac{2}{|1_{\eta_0}|}$.

	We now investigate spectral properties of $T_\mu$ on $\cal_{ji}=\ker(\Delta_j-\alpha_{ji}I)$. 
	It is clear that $\cal_{ji}$ is a finite-dimensional subrepresentation of $L^2(G)$. Since $e_k$'s are also eigenfunctions of $T_{\mu}$, 
	\[
	\lcal(\mu;\cal_{ji})=\min\{-\log\|\mu\ast e_k\|_2 : e_k\in \cal_{ij}\}. 
	\]
Let $\nu=\mu^{(l)}$ for some positive integer $l$ to be specified later, and let $e_k\in\mathcal H_{ij}$; note that $ \alpha_{ jk}= \alpha_{ ji}$. 
By the definition of $\mathcal H_{ij}$ and Lemma~\ref{lem:eigenfunctions-live-at-certain-scale}, $e_k$ lives at scale $\eta_j$. 
Thus we have
\[
\bigl|\|(e_k)_{\eta_{j+2}}\ast \nu\|_2-\|(e_k\ast \nu)\|_2\bigr|\le \|((e_k)_{\eta_{j+2}}-e_k)\ast \nu\|_2\le \eta_{j}^{a/2},
\]
which implies that $|\|(e_k)_{\eta_{j+2}}\ast \nu\|_2^2-\|e_k\ast \nu\|_2^2|\le 2 \eta_{j}^{a/2}$. Therefore, 
\be\label{eq:thickening-given-scale}
\|e_k\ast \nu\|_2^2\le 2 \eta_{j}^{a/2}+\|(e_k)_{\eta_{j+2}}\ast \nu\|_2^2.
\ee	
On the other hand, by the Mixing Inequality (see Theorem~\ref{thm:mixing-inequality}), we have
\begin{align}
\notag	\|(e_k)_{\eta_{j+2}}\ast \nu\|_2^2 = & \|e_k\ast \nu_{\eta_{j+2}}\|_2^2 \\
\notag	\le & 2 \|(e_k)_{\eta_j^{1/a}}\|_2^2 \|(\nu_{\eta_{j+2}})_{\eta_j^{1/a}}\|_2^2+ {\eta_j^{1/(8aL)}} \|\nu_{\eta_{j+2}}\|_2^2 \\
\label{eq:mixing-scale}	\le & (2 \eta_j^{1/a}+{\eta_j^{1/(8aL)}}) \|\nu_{\eta_{j+2}}\|_2^2 
	\le 3 \eta_j^{1/(8aL)}\|\nu_{\eta_{j+2}}\|_2^2
\end{align}
where the second inequality follows from the fact that $e_k$ lives as scale $\eta_j$.

By \eqref{eq:thickening-given-scale} and \eqref{eq:mixing-scale}, for every $k\in I_j$, we have
\begin{align*}
	-2\log \bigl( \|e_k\ast \nu\|_2 \bigr) \ge & -\log  \bigl( 2\eta_j^{a/2}+3 \eta_j^{1/(8aL)}\| \nu_{\eta_{j+2}}\|_2^2  \bigr)  \\
	\ge &  -\log 5 - \log  \bigl( \max (\eta_j^{a/2},\eta_j^{1/(8aL)}\| \nu_{\eta_{j+2}}\|_2^2)  \bigr) .
\end{align*}
For $\eta_0\ll_{L,d_0} 1$ small enough, one obtains 
\be\label{eq:lyapanov-lower-bound}
-2\log(\|e_k\ast \nu\|_2)\ge \min \Bigl( -\frac{1}{3a} \log \eta_{j+2}, -\frac{1}{9a^3L} \log \eta_{j+2}-\log \| \nu_{\eta_{j+2}}\|_2^2 \Bigr).
\ee
By Lemma~\ref{lem:basic-properties-metric-entropy} and the dimension condition, we have 
\be\label{eq:metric-entropy-scale}
|h(G;\eta)-\log(1/|1_{\eta}|)|\ll_{d_0,C_1} 1, \text{ and }
|\log(1/|1_{\eta}|) +d_0 \log \eta|\ll_{d_0,C_1} 1.
\ee
Hence for $\eta_0\ll_{C_0,C_1,L} 1$, by \eqref{eq:lyapanov-lower-bound} and \eqref{eq:metric-entropy-scale} we have 
\begin{align}\notag
-2\log(\|e_k\ast \nu\|_2)\ge & 	 \min\left(\frac{1}{4d_0a} h(G;\eta_j), \frac{1}{10Ld_0a^3} h(G;\eta_j)-\log \| \nu_{\eta_{j+2}}\|_2^2\right) \\ 
\label{eq:LyapunovLowerBound}
\ge & 
 \min \biggl( \frac{1}{4d_0a} h(G;\eta_j), H_2(\nu; \eta_{j+2})- \Bigl(1-\frac{1}{10Ld_0a^3} \Bigr)  h(G;\eta_j) \biggr).
\end{align}
By the assumption for some $l_{j+2}\le C_2 h(G;\eta_{j+2})$, we have \[H_2(\mu^{(l_{j+2})}; \eta_{j+2})\ge \Bigl(1-\frac{1}{20Ld_0a^3} \Bigr)  h(G;\eta_{j+2});\] and so 
by applying the inequality \eqref{eq:LyapunovLowerBound} to $\nu= \mu^{ (l_{j+2}) }$ 
for every $i\in I_j$ we have
\be\label{eq:lower-bound-lyaponov-eigenfunctions-scale}
\lcal(\mu; \cal_{ji})\ge \min \Bigl(\frac{1}{8C_2d_0a}, \frac{1}{40C_2Ld_0a^3} \Bigr) =\frac{1}{40C_2Ld_0a^3}.
\ee

Altogether, \eqref{eq:lower-bound-lyaponov-eigenfunctions-scale} and the definition of $\cal_0$ imply
\[
\lcal(\mu;L^2(G)\ominus \cal_0)\ge \frac{1}{40C_2Ld_0a^3},
\] 
as we claimed. 

Since the group generated by the support of $\mu$ is dense in $G$ and $\dim \cal_0 < \infty$,  it follows that $\lcal(\mu;L^2_0(G) )>0$. 
\end{proof}

The following theorem is a corollary of the proof of Theorem~\ref{thm:large-renyi-entropy}. 

 \begin{theorem}\label{thm:functions-scale-spec-gap}
 	In the above setting, suppose $\mu$ is a symmetric Borel probability measure on $G$, and the group generated by the support of $\mu$ is dense in $G$. Suppose that there exist $C_3>0$, $c>0$, and $ 0< \eta_0<1$ such that for every $\eta\le \eta_0$ and every function $g\in L^2(G)$ which lives at scale $\eta$ there exists $l\le C_3 \log(1/\eta)$ such that 
\[ 	\|\mu^{(l)}\ast g\|_2\le \eta^{c} \|g\|_2.\] 
Then there is a subrepresentation $\cal_0$ of $L^2(G)$ with $\dim \cal_0\le 2C_0\eta_0^{-d_0}$ 
such that 
\[
\lcal(\mu; L^2(G)\ominus \cal_0)\ge \frac{c}{C_3}.
\]
In particular, $\lcal(\mu;G)>0$. 
 \end{theorem}
 
\begin{proof}
Without loss of generality, assume that $\eta_0$ is sufficiently small so that Theorem~\ref{thm:large-renyi-entropy} holds. 
As before, fix $a>\max(4Ld_0,4L+2)$, and for $i \ge 1$, set $\eta_i:=\eta_0^{a^i}$.
Let $\{e_i\}_{i=1}^{\infty}$, the sets $I_j$'s, and $E$ be as in the proof of Theorem~\ref{thm:large-renyi-entropy}. 
Define $\cal_0$ as in that proof as well. 
	
For all $i\in I_j$, $e_i$ is function which lives at scale $\eta_j$. This, together with the assumption, implies that $\|\mu^{(l_{ji})}\ast e_i\|_2\le \eta_j^c$ for some positive integer $l_{ji}\le C_3\log(1/\eta_j)$. Hence
	\[
	C_3\log(1/\eta_j) \lcal(\mu; \cal_{ji})\ge -c\log \eta_j, 
	\]   
	where $\cal_{ji}:=\ker(\Delta_j-\alpha_{ji}I)$. In view of this, we have $\lcal(\mu; L^2(G)\ominus \cal_0)\ge c/C_3$. 
	
	Finally, since the group generated by the support of $\mu$ is dense in $G$ and $\cal_0$ is finite dimensional, it follows that $\lcal(\mu;L^2_0(G) )>0$. 
\end{proof}


\section{Gaining entropy in a multi-scale setting}\label{sec:BG}
The goal of this section is to prove Theorem~\ref{thm:multi-scale-BG}. 
In their seminal work \cite{Bourgain-Gamburd-2-08}, Bourgain and Gamburd proved that, if $X$ and $Y$ are random variables taking values in a finite group $G$, then the R\'{e}nyi entropy of $XY$ will be substantially larger than the average of the R\'{e}nyi entropies of $X$ and $Y,$ unless  there is an algebraic obstruction, see also \cite[Lemma 15]{Varju12}. This type of result had been proved earlier for random variables $X$ and $Y$ that are uniformly distributed in subsets $A$ and $B$, respectively. For abelian groups, this is due to Balog and Szemer\'{e}di \cite{Balog-Szemeredi94} and Gowers \cite{Gowers01}. For general groups, this was proved by Tao \cite{Tao08}. In the same work, Tao also proves a multi-scale version of this result. In this section, we will prove a multi-scale version of the aforementioned result of Bourgain and Gamburd, which can be considered as a weighted version of  \cite{Tao08}. 
Similar results have been proved earlier for some specific groups in \cite{Golsefidy-Varju-12, Benoist-Saxce-16, deSaxceLindentrauss15, BIG-17}. We start by recalling the definition of an approximate subgroup. 

\begin{definition}
For $K \ge 1$, a subset $X$ of a group $G$ is called a $K$-\emph{approximate subgroup} if $X$ is symmetric, that is, $X=X^{-1}$ and there exists $T \subseteq X\cdot X$ with $\# T \le K$, such that  $X\cdot X\subseteq T \cdot X$.
\end{definition}

Recall also that if $X$ is a random variable taking finitely many values, then the R\'{e}nyi entropy (of order $2$) of $X$ is defined by 
\[ H_2(X) = - \log \left( \sum_x {\mathbb{P} }(X=x)^2 \right),  \]
where, here and in what follows $\log$ refers to logarithm in base $2$.  It is easy to see that when $X$ and $Y$ take values in a group $G$ then $H_2(XY) \ge \frac{H_2(X)+ H_2(Y)}{2}$ holds.

\begin{theorem}[Bourgain-Gamburd]\label{thm:BG}
Let $G$ be a finite group and suppose $X$ and $Y$ are two $G$-valued random variables. If 
\[
H_2(XY)\le \frac{H_2(X)+H_2(Y)}{2}+\log K
\]
for some positive number $K \ge 2$, then there exists  $H\subseteq G$ such that:
\begin{enumerate}
	\item (Approximate structure) $H$ is an $O(K^{O(1)})$-approximate subgroup.

	\item (Controlling the order) $|\log (\# H)-H_2(X)|\ll \log K$.
	\item (Almost equidistribution) There are elements $x,y\in G$ such that for all $h\in H$
	$$\bbp(X=xh)\ge K^{-O(1)} (\# H)^{-1}, \qquad \bbp(Y=hy)\ge K^{-O(1)} (\# H)^{-1}.$$
\end{enumerate}
\end{theorem}

More generally, suppose  that $G$ is an arbitrary compact group and $A,B\subseteq G$ are two measurable subsets of positive measure.  The energy of the pair $(A, B)$ is defined by  
\be\label{def:energy}
E(A,B):=\|\bb1_A\ast\bb1_B\|_2^2.
\ee
When $G$ is finite, this reduces to 
\[
E(A,B)=\# Q(A,B)/(\# G)^3,
\]
where
\[
Q(A,B):=\{(a,b,a',b')\in A\times B\times A\times B|\h ab=a'b'\}.
\]

For general compact groups, the notion of  $\eta$\emph{-approximate energy} has been introduced in \cite{Tao08}. We will work with two different metrics on $G^4$: For $(g_i)_{1 \le i \le 4}$  and  $(g'_i)_{1 \le i \le 4}$ in $G^4$, define
\begin{equation}
\begin{split}      
d^+ \big( (g_i)_{1 \le i \le 4}, (g'_i)_{1 \le i \le 4} \big) & :=  \sum_{ 1 \le i \le 4} d( g_i, g'_i),  \quad  \text{and} \\
d \big((g_i)_{1 \le i \le 4}, (g'_i)_{1 \le i \le 4} \big) & := \max_{ 1 \le i \le 4} d( g_i, g'_i).
\end{split}
\end{equation}
For non-empty $A, B \subseteq G$ and $\eta>0$, we let
\be\label{def:approx-energy}
E_{\eta}(A,B):=\ncal_{\eta}(Q_{\eta}(A,B)),
\ee
where
\[
Q_{\eta}(A,B):=\{(a,b,a',b')\in A\times B\times A\times B|\h ab\in (a'b')_{\eta}\}
\]
where $ \ncal_{\eta}$ is computed with respect to $d^+$. 
 
The results of this section are proved under a weaker dimension condition that we now define. 
We say that $(G, d)$ satisfies {\it the dimension condition at scale $\eta$ with parameter $C'$} if 
there exist $C>1$ and $d_0 > 0$ such that  
\[
C^{-1} \eta^{d_0} \le |1_{c\eta}|\le C \eta^{d_0}
\]	
holds for all $c\in [C'^{-1},C']$.

Abusing the notation, for two positive quantities $X$ and $Y$ we write $X \preccurlyeq Y$ if $X/Y$ is bounded from above by an expression of the form $ \Omega^{O(1)}$, where  
$ \Omega= {2^{d_0}C^2}$. If $ X\preccurlyeq Y$ and $Y \preccurlyeq X$, we write $X \approx Y$.

\begin{theorem}[\cite{Tao08}, Theorem 6.10]\label{thm:tao}
Suppose $G$ is a compact group with a fixed bi-invariant metric. Suppose $A,B \subseteq G$ are non-empty. For every  $\eta>0$ 
and $K\succcurlyeq 1$, if $G$ satisfies the dimension condition at scale $\eta$ with parameter $C'$ (which is a large universal constant), and the energy bound 
\[
\emph{\text{(EB)}} \hspace{1cm}
E_{\eta}(A,B)\gg K^{-O(1)} \ncal_{\eta}(A)^{3/2} \ncal_{\eta}(B)^{3/2} 
\]
holds, then there is $H\subseteq G$ such that 
\begin{enumerate}
\item (Approximate structure) $H$ is an $K^{O(1)}$-approximate subgroup.
\item (Controlling the metric entropy) $|h(H;\eta)-\frac{h(A;\eta)+h(B;\eta)}{2}|\leq \log K$.
\item (Large intersection) There are $x,y\in G$, such that $|h(A\cap xH;\eta)-h(A;\eta)|\leq \log K$ and $|h(B\cap Hy;\eta)-h(B;\eta)|\leq \log K$.  	
\end{enumerate}	
\end{theorem}

Theorem~\ref{thm:multi-scale-BG} is both a multi-scale version of Theorem \ref{thm:BG} and a weighted version of Theorem \ref{thm:tao}.

Let $X$ and $Y$ be Borel random variables whose distributions are given by measures $\mu$ and $\nu$, respectively. Let $\mu_{\eta}:=\mu\ast P_{\eta}$ and $\nu_{\eta}:=\nu\ast P_{\eta}$.  The idea of the proof is to approximate $\mu_{\eta}$ and $\nu_{\eta}$ by step functions, and find subsets of $\eta$-neighborhoods of supports of $\mu$ and $\nu$ with large $\eta$-approximate energy. We will then apply Theorem \ref{thm:tao} to finish the proof. 
The following lemma summarizes some of the properties of the function $\mu_{\eta}$.

\begin{lem}\label{lem:distribution-func-changing-scale}
Suppose $G$ is a compact group and $G$ satisfies the dimension condition at scale $\eta$ with parameter $C'$
for some  $C'\gg 1$ (larger than a universal constant). Suppose $\mu$ and $\nu$ are two Borel probability measures on $G$ and $f\in L^2(G)$ is non-negative. Then 
\begin{enumerate}
\item 	For all $y\in x_{\eta}$ and $c\in [C'^{-1},C'-1]$, we have
$
\mu_{c\eta}(y) \preccurlyeq \mu_{(c+1)\eta}(x),
$
and $f_{c\eta}(y)\preccurlyeq f_{(c+1)\eta}(x);$
in particular $\mu_{\eta}(y)\preccurlyeq  \mu_{2\eta}(x) \preccurlyeq \mu_{3\eta}(y)$.
\item For any $\eta,\eta'>0$ and $y\in G$, we have $P_{\eta'}(y)\le \frac{|1_{\eta+\eta'}|}{|1_{\eta'}|} P_{\eta'+\eta}\ast  P_{\eta}(y)$. (see~\cite[Lemma A.5]{BIG-17})
\item For $c\in [(C'-1)^{-1},(C'-1)]$, we have $\|\mu_{c\eta}\|_2\approx \|\mu_{\eta}\|_2$ and $\|f_{c\eta}\|_2\approx \|f_{\eta}\|_2$.
\item $\|\mu_{\eta}\ast \nu_{\eta}\|_2\le \|(\mu\ast \nu)_{\eta}\|_2\preccurlyeq \|\mu_{\eta}\ast \nu_{\eta}\|_2$.
\end{enumerate}
\end{lem}
\begin{proof}
The sequence of inequalities $$\mu_{c\eta}(y)=\frac{\mu(y_{c\eta})}{|1_{c\eta}|}\le \frac{|1_{(c+1)\eta}|}{|1_{c\eta}|}\cdot \frac{\mu(x_{(c+1)\eta})}{|1_{(c+1)\eta}|}	\preccurlyeq \mu_{(c+1)\eta}(x)$$
proves the first claim of part (1). The second claim of (1) is a special case. Part (2) is an easy consequence of the fact that, if $y\in 1_{\eta'}$, then for any $x\in 1_{\eta}$ we have $x^{-1}y\in 1_{\eta'+\eta}$.
 
For part (3), by symmetry we can and will assume that $c> 1$. Note that 
$$\mu_{\eta}(y)=\frac{\mu(y_{\eta})}{|1_{\eta}|}\le \frac{|1_{c \eta}|}{|1_{\eta}|}\cdot \frac{\mu(y_{c\eta}) } {|1_{c\eta}|}	\preccurlyeq \mu_{c\eta}(y).$$
Hence, we have $\|\mu_{\eta}\|_2\preccurlyeq \|\mu_{c\eta}\|_2$, and, in particular,  $\|f_{\eta}\|_2\preccurlyeq \|f_{c\eta}\|_2$ . In order to prove the reverse inequality, note that by 
(2) we have $\mu_{c\eta}\preccurlyeq P_{(c+1)\eta}\ast \mu_{\eta}$ and $f_{c\eta}\preccurlyeq P_{(c+1)\eta}\ast f_{\eta}$. These imply that 
\[
\|\mu_{c\eta}\|_2\preccurlyeq 
\|P_{(c+1)\eta}\ast \mu_{\eta}\|_2\le \|\mu_{\eta}\|_2
\quad \text{ and } \quad 
\|f_{c\eta}\|_2\preccurlyeq 
\|P_{(c+1)\eta}\ast f_{\eta}\|_2\le \|f_{\eta}\|_2.
\]  
Finally, to prove (4), first note that 
$$\|\mu_{\eta}\ast \nu_{\eta}\|_2 = \|  P_{\eta}\ast (\mu\ast\nu)_{\eta} \|_2
\le \|(\mu\ast\nu)_{\eta}\|_2.$$

Part (2) implies that  $P_{\eta}\preccurlyeq P_{2\eta}\ast P_{\eta}$, which, in turn, shows that 
\be\label{eq:thickening-one-vs-two}
(\mu\ast \nu)_{\eta}\preccurlyeq \mu_{2\eta}\ast \nu_{\eta}.
\ee 
On the other hand, using (3) and the fact that $\mu\ast \nu_{\eta}$ is a non-negative function,  we have 
\be\label{eq:thickening-one-vs-two-correcting}
\|\mu_{2\eta}\ast \nu_{\eta}\|_2=\|(\mu\ast\nu_{\eta})_{2\eta}\|_2\approx \|(\mu\ast \nu_{\eta})_{\eta}\|_2=\|\mu_{\eta}\ast \nu_{\eta}\|_2;
\ee
applying \eqref{eq:thickening-one-vs-two} and 
\eqref{eq:thickening-one-vs-two-correcting} we obtain the desired inequality.
\end{proof}

From now on, we will assume that $\mu$ and $\nu$ denote the distributions of the random variables $X$ and $Y$, respectively, and that the inequality
$$H_2(XY;\eta)\le \log K+ \frac{H_2(X;\eta)+H_2(Y;\eta)}{2}$$
holds.  Hence we have
\[\label{eq:non-flattening-l2-norm}
\|(\mu\ast \nu)_{\eta}\|_2\ge K^{-1} \|\mu_{\eta}\|_2^{1/2}\|\nu_{\eta}\|_2^{1/2}. 
\]
By Lemma~\ref{lem:distribution-func-changing-scale} and the above inequality we deduce that
\be\label{eq:non-flattening-l2-norm}
\|\mu_{\eta}\ast \nu_{\eta}\|_2\succcurlyeq K^{-1} \|\mu_{\eta}\|_2^{1/2}\|\nu_{\eta}\|_2^{1/2}.
\ee
By \eqref{eq:Young-ineq}, we have  $\|\mu_{\eta}\ast \nu_{\eta}\|_2\le \min(\|\mu_{\eta}\|_2,\|\nu_{\eta}\|_2)$, which implies
\be\label{eq:l2-norms-are-poly-K-the-same}
K^{-2}\|\mu_{\eta}\|_2\preccurlyeq \|\nu_{\eta}\|_2 \preccurlyeq K^2 \|\mu_{\eta}\|_2.
\ee

To find the desired step function approximation of $\mu_\eta$, we discretize $G$ and then choose subsets of this discrete model according to the value of $\mu_{\eta}$. We fix a maximal $\eta$-separating subset $\ccal$ of $G$. 

As it was mentioned  in Remark \ref{rem:boundedscale}, the proof of Lemma~\ref{lem:basic-properties-metric-entropy} only uses the dimension condition 
 for $\eta,  \eta/2$ and $2\eta$.  Hence for $c\in [(C'/2)^{-1},C'/2]$ we have 

\be\label{eq:basic-metric-entropy}
\ncal_{c\eta}(A)\approx \frac{|A_{\eta}|}{|1_{\eta}|}.
\ee
We partition $\ccal$ according to the value of $\mu_{2\eta}$ as follows:
\be\label{eq:partition-large}
\ccal(\mu;>):=\{x\in \ccal|\h \mu_{2\eta}(x) > K^{10} \|\mu_{\eta}\|_2^2\},
\ee 
\be\label{eq:partition-small}
\ccal(\mu;<):=\{x\in \ccal|\h \mu_{2\eta}(x) < K^{-10} \|\mu_{\eta}\|_2^2\},
\ee
and 
\be\label{eq:partition-middle}
\ccal(\mu;\sim):=\{x\in \ccal|\h K^{-10} \|\mu_{\eta}\|_2^2\le  \mu_{2\eta}(x) \le K^{10} \|\mu_{\eta}\|_2^2\}.
\ee
We also define the following functions: 
\be\label{eq:function-partition}
\mu^{>}_{\eta}:=\bb1_{\ccal(\mu;>)_{\eta}}\cdot \mu_{\eta}, \h\h 
\mu^{<}_{\eta}:=\bb1_{\ccal(\mu;<)_{\eta}}\cdot \mu_{\eta},
\ee
and 
\[
\mu^{\sim}_{\eta}(x):=\begin{cases} 
\mu_{\eta}(x) &\text{ if } x\not\in \ccal(\mu;>)_\eta\cup \ccal(\mu;<)_\eta\\
0 & \text{ otherwise.}	
\end{cases}
\]
And so $\mu_{\eta}(x)\le \mu^>_{\eta}(x)+\mu^<_\eta(x)+\mu^\sim_\eta(x)$, and inequality can possibly occur only in $\ccal(\mu;>)_\eta\cap \ccal(\mu;<)_\eta$.
The functions $\mu^{>}_{\eta}$ and $\mu^{<}_{\eta}$ should be viewed as \emph{tails} of $\mu_{\eta}$ and will now be shown to be negligible. 
\begin{lem}\label{lem:controlling-tails}
In the above setting, $\|\mu^{>}_{\eta}\|_1\preccurlyeq K^{-10}$ and $\|\mu^{<}_{\eta}\|_2\preccurlyeq K^{-5} \|\mu_{\eta}\|_2$.	
\end{lem}
\begin{proof}
	For any $y\in \ccal(\mu;>)_{\eta}$, there is $x\in \ccal(\mu,>)$ such that $y\in x_{\eta}$. Applying part (1) of Lemma~\ref{lem:distribution-func-changing-scale} we have 
	\[
	\mu_{3\eta}(y)\succcurlyeq \mu_{2\eta}(x) > K^{10} \|\mu_\eta\|_2^2.
	\]
On the other hand, by part (3) of Lemma~\ref{lem:distribution-func-changing-scale} we have $\|\mu_\eta\|_2\approx \|\mu_{3\eta}\|_2$. Hence, we have 
	\[
	\|\mu_\eta\|_2^2\succcurlyeq \int_{\ccal(\mu,>)_{\eta}} \mu_{3\eta}(y)^2 \d y\succcurlyeq K^{10} \|\mu_{\eta}\|_2^2 \int_{\ccal(\mu,>)_{\eta}} \mu_{\eta}(y) \d y= K^{10} \|\mu_{\eta}\|_2^2 \|\mu^{>}_{\eta}\|_1,
	\]	
	which implies the first inequality. 
	
	For any $y\in \ccal(\mu,<)_{\eta}$, there is $x\in \ccal(\mu,<)$ such that $y\in x_{\eta}$; and so by part (1) of Lemma~\ref{lem:distribution-func-changing-scale} we have $\mu_{\eta}(y)\preccurlyeq \mu_{2\eta}(x)\le K^{-10} \|\mu_{\eta}\|_2^2$. Therefore
	\[
	\|\mu^{<}_{\eta}\|_2^2=\int_{\ccal(\mu,<)_{\eta}} \mu_{\eta}(y)^2 \d y \preccurlyeq  K^{-10} \|\mu_{\eta}\|_2^2 \int_{\ccal(\mu,<)_{\eta}} \mu_{\eta}(y) \d y\le K^{-10} \|\mu_{\eta}\|_2^2;
	\]
	and the second inequality follows.	
\end{proof}
\begin{corollary}\label{cor:middlepart-has-weighted-energy}
In the above setting, $\|\mu^{\sim}_{\eta}\ast \nu^{\sim}_{\eta}\|_2\ge (2K)^{-1} \|\mu_{\eta}\|_2^{1/2} \|\nu_{\eta}\|_2^{1/2}$ if $K\succcurlyeq 1$.	
\end{corollary}
\begin{proof}
For all $y \in G$, we have $\mu_\eta(y)\ge \mu^\sim_\eta(y)$. By Lemma~\ref{lem:controlling-tails}, and \eqref{eq:l2-norms-are-poly-K-the-same}, we have
\begin{align}
\label{eq:needed-inequalities->-mu}
\|\mu^{>}_{\eta}\ast \nu_{\eta}\|_2\preccurlyeq & K^{-10} \|\nu_{\eta}\|_2 \preccurlyeq K^{-9} \|\mu_{\eta}\|_2^{1/2} \|\nu_{\eta}\|_2^{1/2},\\
\label{eq:needed-inequalities-<-mu}
\|\mu^{<}_{\eta}\ast \nu_{\eta}\|_2\le &\|\mu^{<}_{\eta}\|_2 \preccurlyeq K^{-5} \|\mu_{\eta}\|_2\preccurlyeq K^{-4} \|\mu_{\eta}\|_2^{1/2} \|\nu_{\eta}\|_2^{1/2},
\\
\label{eq:needed-inequalities->-nu}
\|\mu^{\sim}_{\eta}\ast \nu^{>}_{\eta}\|_2\preccurlyeq & K^{-10} \|\mu^{\sim}_{\eta}\|_2\le K^{-10} \|\mu_{\eta}\|_2  \preccurlyeq K^{-9} \|\mu_{\eta}\|_2^{1/2} \|\nu_{\eta}\|_2^{1/2},
\\
\label{eq:needed-inequalities-<-nu}
\|\mu^{\sim}_{\eta}\ast \nu^{<}_{\eta}\|_2\le &\|\mu^{\sim}_{\eta}\|_1\|\mu^{<}_{\eta}\|_2 \preccurlyeq K^{-5} \|\nu_{\eta}\|_2\preccurlyeq K^{-4} \|\mu_{\eta}\|_2^{1/2} \|\nu_{\eta}\|_2^{1/2}.
\end{align}
Hence by the triangle inequality and $\mu_\eta(y)\le \mu^>_\eta(y)+\mu^<_\eta(y)+\mu^\sim_\eta(y)$ we get
\[
\|\mu^{\sim}_{\eta}\ast \nu^{\sim}_{\eta}\|_2 \ge (K^{-1}-\Omega^{O(1)}(2K^{-4}+2K^{-9}))\|\mu_{\eta}\|_2^{1/2} \|\nu_{\eta}\|_2^{1/2}.
\]
For $K \succcurlyeq 1$, the claim follows. 
\end{proof}
We will now apply Corollary \ref{cor:middlepart-has-weighted-energy} to prove that the energy $E_{16\eta}(\ccal^{\sim}(\mu;\eta), \ccal^{\sim}(\nu;\eta))$ is \emph{large}. Using this bound and Theorem \ref{thm:tao}, we deduce Theorem \ref{thm:multi-scale-BG}. 

\begin{lem}\label{lem:approx-energy}
For non-empty sets $A,B\subseteq G$, we have 
\[
E_{\eta/16}(A,B)\preccurlyeq \frac{E({A_{\eta}},{B_{\eta}})}{|1_{\eta}|^{3}} \preccurlyeq E_{6\eta}(A,B).
\]	
\end{lem}
\begin{proof}
By definition $E_{\eta}(A,B):=\ncal_{\eta}(Q_{\eta}(A,B))$ with $d^+$-metric on $G^4$. Hence by Lemma~\ref{lem:basic-properties-metric-entropy} we have 
\[
 E_{\eta}(A,B) \approx \frac{|(Q_{\eta}(A,B))_{\eta}|}{|(1,1,1,1)_{\eta}^+|},
\]	
where $+$ indicates that we are using the $d^+$-metric. Since 
\[
(1,1,1,1)_{\eta/4}\subseteq (1,1,1,1)_{\eta}^+\subseteq (1,1,1,1)_{\eta},
\]
 by $|1_{c\eta}|\approx |1_{\eta}|$ we deduce 
\be\label{eq:approx-energy}
E_{\eta}(A,B) \approx \frac{|(Q_{\eta}(A,B))_{\eta}|}{|1_{\eta}|^4}.
\ee
Based on \eqref{eq:approx-energy}, we will focus on $|Q_{\eta}(A,B)_{\eta}|$ and relate it to energies of thickened sets. First, we will exprees $E(A_{\eta},B_{\eta})$ as the measure of a subset of $G^3$:
\begin{align}
\notag E(A_{\eta},B_{\eta})=& \|\bb1_{A_{\eta}}\ast \bb1_{B_{\eta}}	\|_2^2 \\
\notag
= &\int_G\int_G\int_G \bb1_{A_{\eta}}(x)\bb1_{B_{\eta}}(x^{-1}y)\bb1_{A_{\eta}}(z)\bb1_{B_{\eta}}(z^{-1}y)\h \d x \d z \d y \\
\notag
= & |\{(x,z,y)\in A_{\eta}\times A_{\eta}\times G|\h x^{-1}y\in B_{\eta}, z^{-1}y\in B_{\eta}\}|
\\
\label{eq:energy-thicken-sets}
= & |\{(x,z,t)\in A_{\eta}\times A_{\eta}\times B_{\eta}|\h z^{-1}xt\in B_{\eta}\}|.
\end{align}
Using \eqref{eq:energy-thicken-sets}, we can find an upper bound for $|Q_{\eta}(A,B)_{\eta}|$. We have 
\begin{align}
\notag |Q_{\eta}(A,B)_{\eta}|\le & |\{(x_1,x_2,y_1,y_2)\in A_{\eta}\times A_{\eta}\times B_{\eta}\times B_{\eta}|\h y_2^{-1}x_2^{-1}y_1x_1\in 1_{5\eta} \}| 
\\
\notag = &
|\{(x_1,x_2,y_1,h)\in  A_{\eta}\times A_{\eta}\times B_{\eta}\times 1_{5\eta}|\h x_2^{-1}x_1y_2h^{-1}\in B_{\eta}\}|
\\
\notag \le &
 |\{(x_1,x_2,y_1,h)\in  A_{\eta}\times A_{\eta}\times B_{\eta}\times 1_{5\eta}|\h x_2^{-1}x_1y_2\in B_{6\eta}\}|	
 \\
\notag \preccurlyeq & |1_{\eta}|  |\{(x_1,x_2,y_1)\in  A_{6\eta}\times A_{6\eta}\times B_{6\eta}|\h x_2^{-1}x_1y_2\in B_{6\eta}\}|	
\\
\label{eq:upper-bound-energy} = &|1_{\eta}| E(A_{6\eta},B_{6\eta}).
\end{align}
Again using \eqref{eq:energy-thicken-sets}, we find a lower bound for $|Q_{\eta}(A,B)_{\eta}|$:
\begin{align}
\notag 
|Q_{\eta}(A,B)_{\eta}| \ge & 
|\{(x_1,x_2,y_1,y_2)\in A_{\eta/8}\times A_{\eta/8}\times B_{\eta/8}\times B_{\eta/8}|\h y_2^{-1}x_2^{-1}y_1x_1\in 1_{\eta/2} \}| 
\\
\notag
= &
|\{(x_1,x_2,y_1,h)\in  A_{\eta/8}\times A_{\eta/8}\times B_{\eta/8}\times 1_{\eta/2}|\h x_2^{-1}y_1x_1h^{-1}\in B_{\eta/8}\}|
\\
\notag
\ge 
&
|\{(x_1,x_2,y_1,h)\in  A_{\eta/8}\times A_{\eta/8}\times B_{\eta/8}\times 1_{\eta/16}|\h x_2^{-1}y_1x_1h^{-1}\in B_{\eta/16}\}|
\\
\label{eq:lower-energy}
\succcurlyeq
&
|1_{\eta}| E(A_{\eta/16},B_{\eta/16}).	
\end{align}
By \eqref{eq:approx-energy}, \eqref{eq:upper-bound-energy}, and \eqref{eq:lower-energy}, claim follows.
\end{proof}
\begin{lem}\label{lem:volume-thickenned-middle-part}
In the above setting, $\frac{1}{K^{O(1)}\|\mu_{\eta}\|_2^2}\le |\ccal(\mu,\sim)_\eta| \le \frac{K^{O(1)}}{\|\mu_{\eta}\|_2^2}$.	
\end{lem}
\begin{proof}
For all $y\in \ccal(\mu,\sim)_\eta$, there exists $x\in \ccal(\mu,\sim)$ such that $y\in x_\eta$. Hence by part (1) of Lemma~\ref{lem:distribution-func-changing-scale} we have
\[
\mu_{3\eta}(y) \succcurlyeq \mu_{2\eta}(x) \succcurlyeq K^{-20} \|\mu_{\eta}\|_2^2,
\]
which implies that 
\[
\|\mu_{3\eta}\|_2^2 \succcurlyeq K^{-20} \|\mu_{\eta}\|_2^4|\ccal(\mu, \sim)_\eta|.
\]
Therefore by part (3) of Lemma~\ref{lem:distribution-func-changing-scale} we deduce that 
\[
|\ccal(\mu, \sim)_\eta|\preccurlyeq \frac{K^{20}}{\|\mu_{\eta}\|_2^2}.
\]
It follows from the definition of $\mu^\sim_\eta$ that the support of $\mu^\sim_\eta$ is a subset of $\ccal(\mu,\sim)_{\eta}$. Hence if 
$\mu^\sim_\eta(y)\neq 0$, then there is $x\in \ccal(\mu,\sim)$ such that $y\in x_{\eta}$. So, by part (1) of Lemma~\ref{lem:distribution-func-changing-scale} we have 
\be\label{eq:infinity-norm-middle-weighted}
\mu_\eta(y)\preccurlyeq \mu_{2\eta}(x)\le K^{10}\|\mu_\eta\|_2^2, \text{ which implies } \|\mu^\sim_\eta\|_\infty\preccurlyeq K^{10}\|\mu_\eta\|_2^2.
\ee
Therefore we get
\be\label{eq:upper-bound-volume-middle-set}
\|\mu^\sim_\eta\|_2^2\le \|\mu^\sim_\eta\|^2_\infty |\ccal(\mu,\sim)_{\eta}|\preccurlyeq K^{20}\|\mu_\eta\|_2^4 |\ccal(\mu,\sim)_{\eta}|.
\ee
By \eqref{eq:l2-norms-are-poly-K-the-same}, Corollary~\ref{cor:middlepart-has-weighted-energy}, and \eqref{eq:upper-bound-volume-middle-set}, we get
\begin{align*}
K^{-2} \|\mu_{\eta}\|_2^2 \preccurlyeq & \|\mu_\eta\|_2\|\nu_\eta\|_2 \preccurlyeq K^2\|\mu^\sim_{\eta}\ast \nu^\sim_{\eta}\|_2^2
\\
\le & K^2 \|\mu^\sim_{\eta}\|_2^2 \preccurlyeq K^{22} \|\mu_\eta\|_2^4 |\ccal(\mu,\sim)_{\eta}|;
\end{align*}
Therefore
\[
\frac{1}{K^{24}\|\mu_\eta\|_2^2}\preccurlyeq |\ccal(\mu,\sim)_{\eta}|;
\]
and the claim follows.
\end{proof}
\begin{proposition}\label{prop:energy-middle-sets}
In the above setting the inequality
\[ E_{16\eta}(\ccal(\mu;\sim),\ccal(\nu;\sim))\succcurlyeq \frac{1}{K^{O(1)}} \ncal_{16\eta}(\ccal(\mu;\sim))^{3/2}\ncal_{16\eta}(\ccal(\nu;\sim))^{3/2}
\]
holds, where $\ccal(\mu;\sim)$ is defined in \eqref{eq:partition-middle}. 
\end{proposition}
\begin{proof}
By \eqref{eq:infinity-norm-middle-weighted}, we have 
\[
\mu^\sim_\eta\preccurlyeq (K^{10} \|\mu_\eta\|_2^2)\h \bb1_{\ccal(\mu,\sim)_\eta}
\text{ and }
\nu^\sim_\eta\preccurlyeq (K^{10} \|\nu_\eta\|_2^2)\h  \bb1_{\ccal(\nu,\sim)_\eta}.
\]	
It follows that 
\[ \|  \mu^\sim_\eta  \ast \nu^\sim_\eta  \|^2_2 \preccurlyeq  K^{40} \|\mu_\eta\|_2^4 \|\nu_\eta\|_2^4  \| \bb1_{\ccal(\mu,\sim)_\eta} \ast \bb1_{\ccal(\nu,\sim)_\eta} \|_2^2 =
K^{40} \|\mu_\eta\|_2^4 \|\nu_\eta\|_2^4 E(\ccal(\mu,\sim)_\eta,\ccal(\nu,\sim)_\eta). \]
By Corollary~\ref{cor:middlepart-has-weighted-energy} and the above inequality we have
\be\label{eq:towards-lower-bound-energy}
K^{-2} \|\mu_\eta\|_2\|\nu_\eta\|_2 \preccurlyeq K^{40} \|\mu_\eta\|_2^4 \|\nu_\eta\|_2^4 E(\ccal(\mu,\sim)_\eta,\ccal(\nu,\sim)_\eta).
\ee
By Lemma~\ref{lem:volume-thickenned-middle-part} and \eqref{eq:towards-lower-bound-energy}, we obtain
\be\label{eq:lower-bound-energy-1}
K^{-O(1)} |\ccal(\mu,\sim)_\eta)|^{3/2} |\ccal(\nu,\sim)_\eta)|^{3/2}\preccurlyeq E(\ccal(\mu,\sim)_\eta,\ccal(\nu,\sim)_\eta);
\ee
and so by Lemma~\ref{lem:basic-properties-metric-entropy} and Lemma~\ref{lem:approx-energy}, we deduce
\[
K^{-O(1)} \ncal_{16\eta}(\ccal(\mu,\sim))^{3/2} \ncal_{16\eta}(\ccal(\nu,\sim))^{3/2} 
\preccurlyeq E_{16\eta}(\ccal(\mu,\sim),\ccal(\nu,\sim));
\]
and the claim follows.
\end{proof}
\begin{proof}[Proof of Theorem~\ref{thm:multi-scale-BG}]
Recall that $\mu$ and $\nu$ denote the distribution measures of random variables $X$ and $Y$, respectively, and $Z$ denotes a random variable independent of $X$ and $Y$ with uniform distribution over $1_{ 3 \eta}$.

By Proposition~\ref{prop:energy-middle-sets}, for $K\succcurlyeq 1$, 
we can apply Theorem~\ref{thm:tao} to the sets $A= \ccal(\mu;\sim)$ and $B= \ccal(\nu;\sim)$ to obtain  $H\subseteq G$ and $x,y\in G$ such that
\begin{enumerate}
\item (Approximate structure) $H$ is an $K^{O(1)}$-approximate subgroup.
\item (Controlling the metric entropy) $|h(H;16\eta)-\frac{h(\ccal(\mu;\sim);16\eta)+h(\ccal(\nu;\sim);16\eta)}{2}| \leq \log K$. 	
\item (Large intersection) $|h(\ccal(\mu;\sim)\cap xH;16\eta)-h(\ccal(\mu;\sim);16\eta)| \leq\log K$ and 
$$|h(\ccal(\nu;\sim)\cap Hy;16\eta)-h(\ccal(\nu;\sim);16\eta)| \leq \log K.$$
\end{enumerate}
We will show that Theorem \ref{thm:multi-scale-BG} holds for these choices of $H \subseteq G$ and $x, y \in G$. 

By Lemma~\ref{lem:basic-properties-metric-entropy} we have $|\log \ncal_{16\eta}(\ccal(\mu;\sim))-\log(|\ccal(\mu;\sim)_\eta|/|1_\eta|)|\preccurlyeq 1$. Hence, Lemma~\ref{lem:volume-thickenned-middle-part} implies 
\[
|\log \ncal_{16\eta}(\ccal(\mu;\sim))-(\log(1/|1_\eta|)-\log \|\mu_\eta\|_2^2)| \ll \log K
\]
if $K\succcurlyeq 1$. Thus
\be\label{eq:metric-entropy-renyi-entropy}
|\log \ncal_{16\eta}(\ccal(\mu;\sim))-H_2(\mu;\eta)|\ll \log K.
\ee
By \eqref{eq:metric-entropy-renyi-entropy}, Lemma~\ref{lem:basic-properties-metric-entropy}, and part (2) of Theorem \ref{thm:tao} we have 
\[
\left|h(H;\eta)-\frac{H_2(\mu;\eta)+H_2(\nu;\eta)}{2}\right|\ll \log K
\]
if $K\succcurlyeq 1$. We also notice that by \eqref{eq:l2-norms-are-poly-K-the-same} we have $|H_2(\mu;\eta)-H_2(\nu;\eta)|\ll \log K$. Combining these two fact we deduce that
\[
|h(H;\eta)-H_2(\mu;\eta)|\ll \log K.
\] 
This proves the second property mentioned in Theorem \ref{thm:multi-scale-BG} for the set $H$. 

Finally, to prove the third property, note that 
\[ \ncal_\eta(\ccal(\mu;\sim)\cap xH)\succcurlyeq K^{-O(1)} \ncal_{\eta}(\ccal(\mu;\sim));
\]
and so by \eqref{eq:metric-entropy-renyi-entropy} we get
\begin{equation}\label{eq:boundNH}
\ncal_\eta(\ccal(\mu;\sim)\cap xH) \succcurlyeq  K^{-O(1)} 2^{H_2(\mu;\eta)}.
\end{equation}
On the other hand, by Lemma~\ref{lem:basic-properties-metric-entropy}, Corollary~\ref{cor:larger-nbhd}, and the fact that $\ccal(\mu;\sim)$ is an $\eta$-separated set, we have 
\[
\ncal_\eta(\ccal(\mu;\sim)\cap xH)\approx \ncal_{\eta/2}(\ccal(\mu;\sim)\cap xH)=  \#(\ccal(\mu;\sim)\cap xH).
 \]
 Altogether we have
\begin{equation}\label{eq:tedaad}
\#(\ccal(\mu;\sim)\cap xH)\succcurlyeq K^{-O(1)} 2^{H_2(\mu;\eta)}.
\end{equation}

For every $z'\in \ccal(\mu;\sim)_\eta$ there exist $z\in \ccal(\mu;\sim)$ such that $z'\in z_\eta$.  Since $\mu_{3 \eta}(z')=\mu(z'_{3\eta})/|1_{3\eta}|$
and $\mu_{2\eta}(z)\ge K^{-10}\|\mu_{\eta}\|_2^2$,  by part (1) of Lemma \ref{lem:distribution-func-changing-scale} we have
\begin{equation}\label{eq:karaan}
\mu_{3\eta}(z')\succcurlyeq \mu_{2\eta}(z) \ge K^{-10}\|\mu_{\eta}\|_2^2, \quad \text{ and }  \quad \mu(z'_{3\eta}) \ge \widehat{C} K^{-10} 2^{-H_2(\mu;\eta)}.
\end{equation}
where $ \widehat{C}= \Omega^{O(1)}$. 
 Therefore
 \begin{align*}
\bbp(XZ\in (xH)_{\eta})  \ge &  \int_{(\ccal(\mu;\sim)\cap xH)_\eta} \mu_{3\eta}(z')\d z' \\
\succcurlyeq &
  K^{-10}\|\mu_{\eta}\|_2^2 |(\ccal(\mu;\sim)\cap xH)_\eta| \\
  \approx &  K^{-10} 2^{-H_2(\mu;\eta)} \ncal_{\eta}(\ccal(\mu;\sim)\cap xH) \succcurlyeq   K^{-O(1)}.	
 \end{align*}
The lower bound for $\bbp(ZY\in (Hy)_{\eta})$ can be proved by a similar argument. Finally, to prove the last claim, we have
\begin{align*}
| \{h\in H_{\eta} | \h \bbp(X\in (xh)_{3\eta})  \ge \widehat{ C} K^{-10} 2^{-H_2(X;\eta)}\}  | & =
| \{ z' \in (xH)_{\eta} |   \mu( z'_{ 3 \eta})   \ge \widehat{ C} K^{-10} 2^{-H_2(X;\eta)}  \} |  \\
& \ge  | \{ z' \in (\ccal(\mu;\sim)\cap xH)_{\eta} |   \mu( z'_{ 3 \eta})   \ge \widehat{ C} K^{-10} 2^{-H_2(X;\eta)}  \} | \\
& =  | (\ccal(\mu;\sim)\cap xH)_{\eta}| \\
& \succcurlyeq K^{-O(1)} 2^{H_2(\mu;\eta)} \cdot |1_{\eta}| = K^{-O(1)} | H_{\eta} |.
 \end{align*}
This proves the claim.  
\end{proof}

\bibliographystyle{plain}
\bibliography{Ref}
\end{document}